\documentclass[reqno, a4paper, 10pt]{amsart}

\usepackage{etex}

\usepackage{amsmath}
\usepackage{amsthm}
\usepackage{mathtools}
\usepackage{amssymb}
\usepackage{alltt}

\usepackage{graphicx}
\usepackage{color}
\usepackage{comment}
\usepackage{url}
\usepackage{enumitem}
\usepackage{hyperref}
\usepackage[foot]{amsaddr}

\usepackage{pgf,tikz}
\usetikzlibrary{arrows,spy}
\usepackage{pgfplots}
\definecolor{darkgreen}{cmyk}{1.,0.,1.,0.2}

\usepackage[percent]{overpic}
\usepackage{subcaption}
\usepackage{tabu}
\usepackage{microtype}

\usepackage[english]{babel}
\usepackage[utf8]{inputenc}

\usepackage[sort,nocompress]{cite} 

\newcommand{\vectornorm}[1]{\left|\left|#1\right|\right|}
\newcommand{\rr}[0]{\mathbb{R}}
\newcommand{\zz}[0]{\mathbb{Z}}

\newcommand{\const}[0]{\text{const}}
\newcommand{\cover}[1]{\xRightarrow{\protect\mathmakebox[1.7em]{#1}}}

\newcommand{\longlongcover}[1]{\xRightarrow{\protect\mathmakebox[3.9em]{#1}}}

\newcommand{\backcover}[1]{\xLeftarrow{\mathmakebox[1.7em]{#1}}}

\newcommand{\gencover}[1]{\xLeftrightarrow{\mathmakebox[1.7em]{#1}}}
\newcommand{\longgencover}[1]{\xLeftrightarrow{\mathmakebox[2.3em]{#1}}}

\DeclareMathOperator{\sgn}{sgn}
\DeclareMathOperator{\inter}{int}

\DeclareMathOperator{\id}{id}
\DeclareMathOperator{\diam}{diam}

\DeclareMathOperator{\conv}{conv}

\newtheorem{thm}{Theorem}[section]
\newtheorem{cor}[thm]{Corollary}
\newtheorem{lem}[thm]{Lemma}
\newtheorem{prop}[thm]{Proposition}

\theoremstyle{remark}
\newtheorem{rem}[thm]{Remark}

\theoremstyle{definition}
\newtheorem{defn}[thm]{Definition}

\theoremstyle{remark}

\title[Periodic solutions of the FHN equations]{Existence of Periodic Solutions of the FitzHugh-Nagumo Equations for An Explicit Range of the Small Parameter$^*$}

\begin{document}

\author[A. Czechowski]{Aleksander Czechowski$^\dagger$}
\author[P. Zgliczy\'nski]{Piotr Zgliczy\'nski$^\dagger$}
\thanks{$^*$AC was supported by the Foundation for Polish Science under the MPD Programme ‘‘Geometry and Topology in Physical Models’’,
  co-financed by the EU European Regional Development Fund, Operational Program Innovative Economy 2007-2013. PZ was supported by Polish
  National Science Centre grant 2011/03B/ST1/04780.
}

\thanks{$^\dagger$Institute of Computer Science and Computational Mathematics, Jagiellonian University, \L ojasiewicza 6, 30-348 Krak\'ow, Poland
  (\href{mailto:czechows@ii.uj.edu.pl}{czechows@ii.uj.edu.pl}, \href{mailto:zgliczyn@ii.uj.edu.pl}{zgliczyn@ii.uj.edu.pl}).}

\maketitle

\begin{footnotesize}

\textrm{\textbf{Abstract.} 
The FitzHugh-Nagumo model describing propagation of nerve impulses in axon 
is given by fast-slow reaction-diffusion equations, with dependence on a parameter $\epsilon$ representing the ratio of time scales.
It is well known that for all sufficiently small $\epsilon>0$ the system possesses a periodic traveling wave.
With aid of computer-assisted rigorous computations, we
prove the existence of this periodic orbit in the traveling wave equation for an explicit range $\epsilon \in (0, 0.0015]$.
Our approach is based on a novel method of combination of topological techniques of covering relations and isolating segments,
for which we provide a self-contained theory.
We show that the range of existence is wide enough, so the upper bound can be reached by standard validated continuation procedures.
In particular, for the range $\epsilon \in [1.5 \times 10^{-4}, 0.0015]$ we perform a rigorous continuation
based on covering relations and not specifically tailored to the fast-slow setting.
Moreover, we confirm that for $\epsilon=0.0015$ the classical interval Newton-Moore method applied to a sequence of Poincar\'e maps already succeeds.
Techniques described in this paper can be adapted to other fast-slow systems of similar structure.}

\bigskip

\textrm{\textbf{Key words.} fast-slow system, periodic orbits, rigorous numerics, FitzHugh-Nagumo model, isolating segments, covering relations}

\textrm{ \textbf{AMS subject classifications.} 34C25, 34E13, 65G20}
\end{footnotesize}

\section{Introduction}

\subsection{The FitzHugh-Nagumo equations}
The FitzHugh-Nagumo model with diffusion
\begin{equation}\label{RDiff}
  \begin{aligned}
    \frac{\partial u}{\partial \tau} &= \frac{1}{\gamma} \frac{\partial^{2} u}{\partial x^{2}} + u(u-a)(1-u) - w, \\
    \frac{\partial w}{\partial \tau} &= \epsilon (u - w),
 \end{aligned}
\end{equation}
was introduced as a simplification of the Hodgkin-Huxley model
for the nerve impulse propagation in nerve axons~\cite{FitzHugh, Nagumo}.
The variable $u$ represents the axon membrane potential and $w$ a slow negative feedback.
Traveling wave solutions of~\eqref{RDiff} are of particular interest as they resemble an actual motion of the nerve impulse~\cite{Hastings}.
By plugging the traveling wave ansatz $(u,w)(\tau ,x) = (u,w)(x+\theta \tau) = (u,w)(t)$, $\theta > 0$
and rewriting the system as a set of first order equations we arrive at an ODE
\begin{equation}\label{FhnOde}
     \begin{aligned}
       u'&=v, \\
       v'&=\gamma(\theta v - u(u-a)(1-u) + w), \\
       w'&= \frac{\epsilon}{\theta} (u - w).
     \end{aligned}
\end{equation}
to which we will refer to as the \emph{FitzHugh-Nagumo system} or the \emph{FitzHugh-Nagumo equations}.
The FitzHugh-Nagumo system is a fast-slow system with two fast variables $u,v$ and one slow variable $w$. 
The parameter $\theta$ represents the wave speed and $\epsilon$ is the small parameter, so $0 < \epsilon \ll 1$.
To focus our attention, following \cite{ArioliKoch, Champneys, GuckenheimerKuehn} we set the two remaining parameters to
\begin{align}\label{eq:parameters}
  a := 0.1, \qquad  \gamma := 0.2,
\end{align}
throughout the rest of the paper.

Bounded solutions of~\eqref{FhnOde} yielding different wave profiles have been studied by many authors both rigorously and numerically, see
\cite{Conley, Carpenter, Hastings2,
Hastings3, Maginu, Mischaikow, Kopell, Ambrosi, ArioliKoch, Jones, Yanagida, Nagumo, Champneys, GuckenheimerKuehn} and references given there.
Periodic orbits leading to \emph{periodic wave trains} exist for an open range of $\theta$'s and were treated in \cite{Conley,Hastings2,Hastings3,Maginu,Mischaikow, ArioliKoch};
\emph{traveling pulses} generated by homoclinic orbits exist for two isolated values of $\theta$, their existence was proved in \cite{Carpenter,Hastings3, Kopell, Ambrosi, ArioliKoch}.
Stability of waves was discussed in \cite{Jones,Maginu,Yanagida, ArioliKoch}.
Proofs of existence use various methods, but most share the same perturbative theme\footnote{In~\cite{Ambrosi, ArioliKoch} the authors
use non-perturbative computer-assisted methods for a single value $\epsilon=0.01$ where the system becomes a regular, although stiff ODE.
It is to be noted that in~\cite{ArioliKoch} authors prove stability of the wave for this particular parameter.}.
We outline it below, for the periodic orbit.

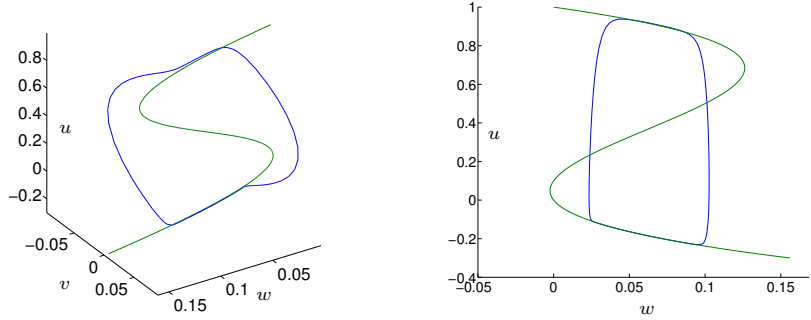
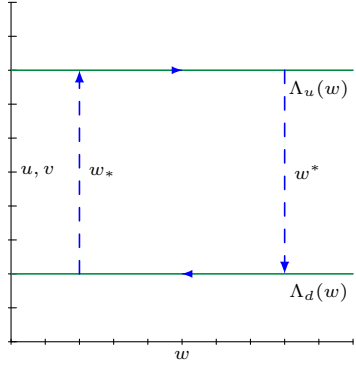
\begin{figure}[t]
  \centering
  \begin{subfigure}{0.48\textwidth}
    \centering
    \begin{overpic}[width=0.7\textwidth, clip=true, trim = 60mm 20mm 100mm 0mm]{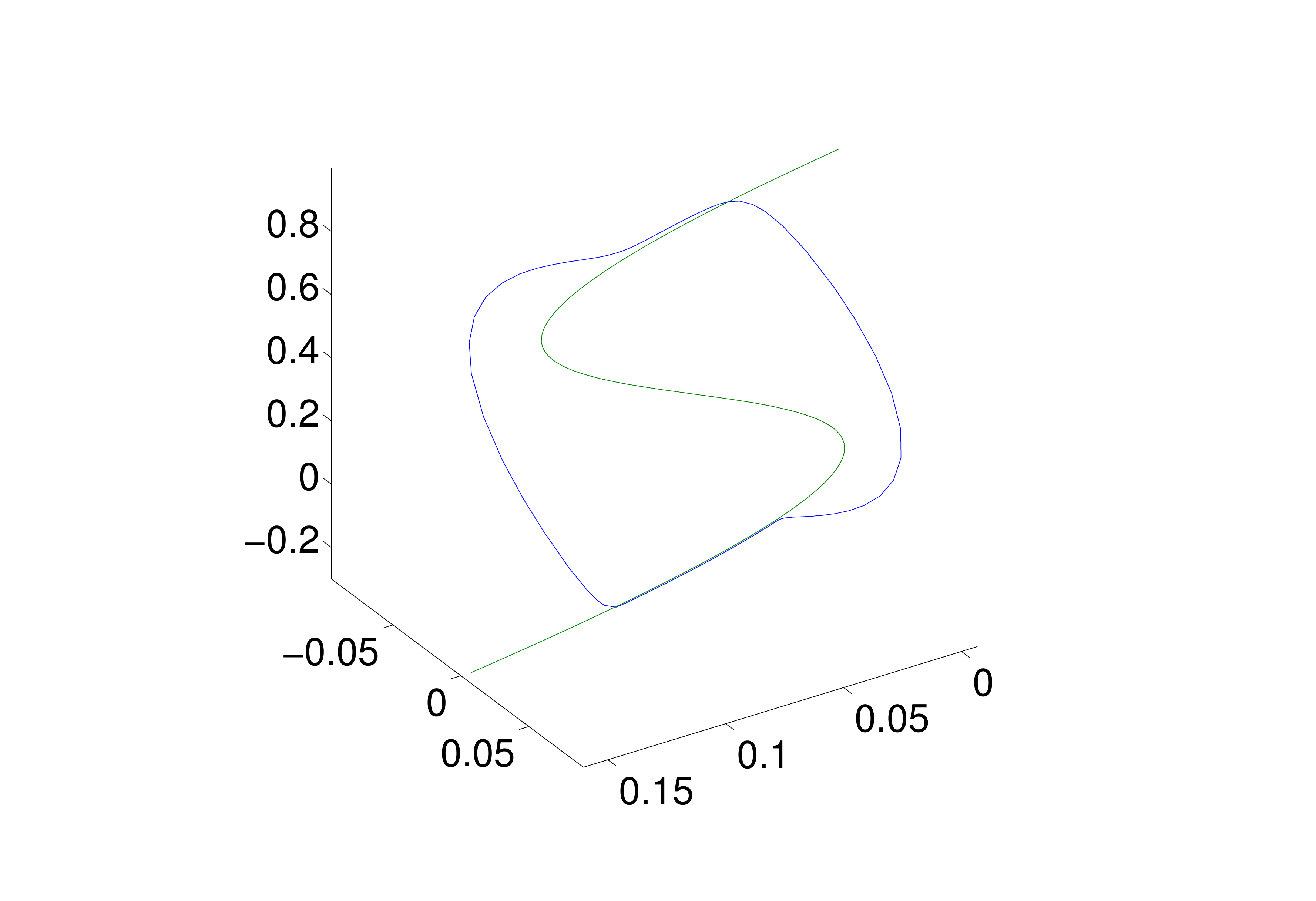}
      \put (17,52) {\scriptsize{$u$}}
      \put (17,9) {\scriptsize{$v$}}
      \put (71,6){\scriptsize{$w$}}
    \end{overpic}
    \caption{Approximate periodic orbit for $\epsilon=0.001$ in blue.}
  \end{subfigure} 
  \begin{subfigure}{0.48\textwidth}
    \centering
    \begin{overpic}[width=0.78\textwidth, clip=true, trim = 25mm 0mm 70mm 10mm]{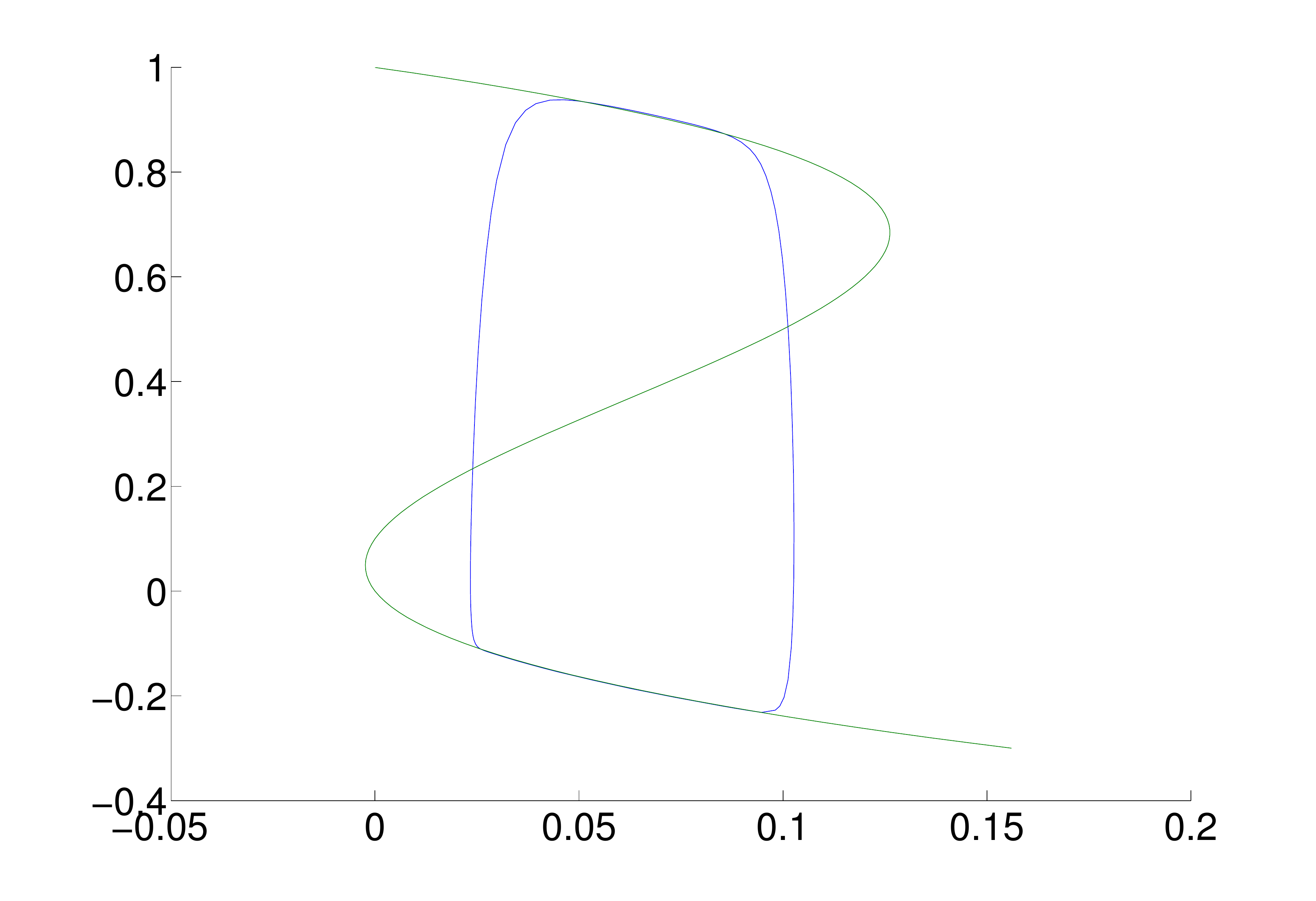}
      \put (11,49) {\scriptsize{$u$}}
      \put (53, 0) {\scriptsize{$w$}}
    \end{overpic}
    \caption{Projection onto $(u,w)$ plane.}
  \end{subfigure}
  \\
  \begin{subfigure}{0.45\textwidth}
    \centering
      \begin{tikzpicture}[line cap=round,line join=round,>=latex,x=4.5mm,y=4.5mm]
 
  \draw[color=black] (-0.,0.) -- (10,0.);
  \foreach \x in {0,...,10}
  \draw[shift={(\x,0)},color=black] (0pt,1pt) -- (0pt,-1pt) node[below] {};

  \draw[color=black] (0.,0.) -- (0.,10);
  \foreach \y in {0,...,10}
  \draw[shift={(0,\y)},color=black] (2pt,0pt) -- (-1pt,0pt) node[left] {};
 

  \clip(0,-1) rectangle (10,10);

  \draw [color=darkgreen,semithick,domain=0.:10.] plot(\x,{8});
  \draw [color=darkgreen,semithick,domain=0.:10.] plot(\x,{2});

  \draw [->,semithick,dash pattern=on 5pt off 5pt,color=blue] (2.,2.) -- (2.,8.);
  \draw [->,semithick,dash pattern=on 5pt off 5pt,color=blue] (8.,8.) -- (8.,2.);

  \draw [->,color=blue] (4.99,8.) -- (5.01,8.);
  \draw [->,color=blue] (5.01,2.) -- (4.99,2.);

  \draw (9.,8.) node[anchor=north] {\scriptsize{$\Lambda_u(w)$}};
  \draw (9.,2.) node[anchor=north] {\scriptsize{$\Lambda_d(w)$}};

  \draw (2.,5.) node[anchor=west] {\scriptsize{$w_{*}$}};
  \draw (8.,5.) node[anchor=west] {\scriptsize{$w^{*}$}};

  \draw (5.,0.) node[anchor=north] {\scriptsize{$w$}};
  \draw (0.,5.) node[anchor=west] {\scriptsize{$u,v$}};

\end{tikzpicture}
      \caption{A schematic drawing of the singular periodic orbit.}
  \end{subfigure}
  \caption{A numerical approximation of the periodic orbit close to the singular orbit, the slow manifold in green.}
  \label{orbitFig}
\end{figure}

Consider the limit equation at $\epsilon = 0$. There, the velocity of $w$
is zero and the phase space can be fibrated into a family of two-dimensional \emph{fast subsystems} parameterized by $v$.
These subsystems serve as a good approximation to the system with $\epsilon > 0$ small, except for regions of phase
space near the Z-shaped \emph{slow manifold}
\begin{equation}
  C_{0} = \{(u,v,w):\ v=0,\ w = u(u-0.1)(1-u) \},
\end{equation}
where the velocities of fast variables become small and the slow variable takes over.
There, one can analyze the \emph{slow flow} given by evolving $w$ while keeping the restriction $(u,v) \in C_{0}$, obtaining a differential algebraic equation.
Note that the slow manifold is formed by fixed points of the fast subsystems. For a range of $w$ it
has exactly three branches - by looking from a perspective of the respective $u$ values the lower and the upper one are formed by saddles,
and the middle one is formed by sources. We denote the upper/lower branches of saddles by $\Lambda_u(w)$ and $\Lambda_d(w)$, respectively.
For exactly two values $w \in \{w_{*}, w^{*} \}$, with $w_{*} < w^{*}$ there are heteroclinic connections from $\Lambda_d(w_{*})$ to $\Lambda_u(w_{*})$
and from $\Lambda_u(w^{*})$ to $\Lambda_d(w^{*})$.
To see that it is best to regard the fast subsystem as a Hamiltonian system with added negative friction and shoot with $w$,
a detailed description is given in~\cite{Conley}.
It happens that in the range $[w_{*},w^{*}]$ the slow flow on the branch $\Lambda_u$ is monotonically decreasing
and on the branch $\Lambda_d$ monotonically increasing, so by connecting the heteroclinics with pieces of the slow manifold
one assembles the \emph{singular periodic orbit}, see Figure~\ref{orbitFig}.
The proof of existence of an actual periodic orbit goes by perturbing to $\epsilon>0$ small and using
certain arguments based on topological methods~\cite{Conley,Carpenter,Mischaikow} or Fenichel theory and differential forms~\cite{Kopell, JonesBook}.

Using methods described above proofs of existence have been given for $\epsilon \in (0,\epsilon_0]$, $\epsilon_0$ ``small enough''.
With aid of computer we are able to improve these results.
We give an \emph{explicit} $\epsilon_{0}$ such that a periodic orbit of~\eqref{FhnOde} exists for the parameter range $\epsilon \in (0,\epsilon_{0}]$.
Our secondary objective is to make $\epsilon_{0}$ as large as possible, so there is no doubt that for $\epsilon \geq \epsilon_{0}$ one can perform
further continuation using well-established computer-aided methods such as the interval Newton-Moore method~\cite{Alefeld, Neumaier, Moore}
applied to a sequence of Poincar\'e maps.
The main results of this paper are:
\begin{thm}\label{thm:main1}
  For each $\epsilon \in (0,1.5 \times 10^{-4}]$, for $\theta = 0.61$ and other parameter values given in~\eqref{eq:parameters} there exists a periodic orbit of~\eqref{FhnOde}.
\end{thm}

\begin{thm}\label{thm:main2}
  For each $\epsilon \in [1.5 \times 10^{-4},0.0015]$, for $\theta = 0.61$ and other parameter values given in~\eqref{eq:parameters} there exists a periodic orbit of~\eqref{FhnOde}.
\end{thm}

\begin{thm}\label{thm:main3}
  For $\epsilon = 0.0015$, $\theta = 0.61$ and other parameter values given in~\eqref{eq:parameters} there exists a periodic orbit of~\eqref{FhnOde},
  which is formed from a locally unique fixed point of a Poincar\'e map.
\end{thm}

The reason we do not merge statements of Theorems~\ref{thm:main1} and~\ref{thm:main2} is a significant difference in proof techniques.
For the proof of Theorem~\ref{thm:main1} we exploit the fast-slow structure and construct a sequence of \emph{isolating segments} and \emph{covering relations}
around the singular orbit.
Let us remark that combination of these two topological methods is new, and isolating segments have been previously applied
only to time-dependent equations.
For Theorem~\ref{thm:main2} we perform a ``regular ODE'' type of proof,
with a parameter continuation method based on verifying covering relations around an approximation of the periodic orbit.
Finally, in Theorem~\ref{thm:main3} we are far enough from $\epsilon=0$, so that a proof by the interval Newton-Moore method applied to a sequence of Poincar\'e maps 
succeeds (see Subsection~\ref{subsec:thm13}), establishing both the existence and local uniqueness.

The motivation for this choice of wave speed was that in numerical simulations parts of the periodic orbit
near the slow manifold stretched relatively long, which allowed us to fully exploit its hyperbolicity.
In the program files a lot of values were hardcoded for this particular $\theta$, but we report that
by substituting $\theta=0.53$, $\theta=0.47$, $\theta \in [0.55,0.554]$ we were also able
to produce results like Theorem~\ref{thm:main1}, for a (shorter) range of $\epsilon \in (0, 5 \times 10^{-5}]$.
We think that by spending time tuning the values in the proof, the range of $\epsilon$ for these $\theta$'s could have been made wider.
This is of course the very same orbit and if one had enough patience, then continuation in $\theta$ would be theoretically possible.

Let us observe that proofs employing \emph{interval arithmetics} are fairly easy to adapt to compact parameter ranges. 
A \emph{validated continuation} can be performed by
subdividing the parameter interval finely and feeding the program with small parameter intervals instead of one exact value~\cite{Day,Lessard,Galias2,WilczakZgliczynski}.
However, we emphasise again that
when trying to prove Theorems~\ref{thm:main1} we are dealing with a half-open range and such straightforward approach 
is bound to fail -- as $\epsilon \to 0^{+}$ the period of the periodic orbit grows to infinity 
and in the singular limit $\epsilon=0$ the orbit is destroyed.
Therefore not all premises in theorems implying existence of such orbits can be verified for $\epsilon \in [0,\epsilon_0]$;
the assumptions need to be formulated in such a way, that the ones which are computationally difficult are possible to check with computer aid for $\epsilon \in [0,\epsilon_0]$,
and the leftover ones are in a simple form where $\epsilon$ can be factored out ``by hand'', assuming $\epsilon > 0$.
To be more specific, $\epsilon$ is factored out to check the monotonicity of the slow flow in
certain sets (the aforementioned isolating segments) in the neighborhood of the slow manifold,
which reduces to a simple inequality for the slow part of the vector field (inequality (S1a) in Remark~\ref{alternativeS}, cf. inequality~\eqref{eq:sgna}).
Clearly, such technique can be applied to other fast-slow systems with one-dimensional slow manifold,
even if the slow flow is nonlinear.

Full proofs are executed with computer assistance and described 
in detail in Section~\ref{sec:implementation}.
For the computer assisted assumption verification we use the CAPD library~\cite{CAPD}, 
which provides algorithms for enclosures of solutions of ODEs, Poincar\'e maps and their derivatives.
By a rigorous enclosure we mean that the true result of a given operation is always contained in an interval object (vector, matrix) returned by the program procedure.
Since our assumptions are in essence a collection of strict inequalities, if we keep the overestimates small,
we should be able to verify them on a machine.

The code which executes the necessary computations is available at the author's homepage~\cite{Czechowski}.
Below we outline the basic ideas of each proof.

\subsection{Outline of the proof of Theorem~\ref{thm:main1}}

We conduct a phase space proof based on a reduction to a sequence of Poincar\'e sections and a fixed point argument for a sequence of Poincar\'e maps.
For $\epsilon > 0$ small the orbit switches between two regimes - the fast one close to heteroclinics of the fast subsystem
and the slow one along the branches $\Lambda_d, \Lambda_u$ of the slow manifold. The strategy is to form a closed sequence of
covering relations and isolating segments and deduce the existence of a fixed point of a sequence of Poincar\'e maps
via a topological theorem -- Theorem~\ref{thm:2}. 

For the fast regime we employ the rigorous integration and we check covering relations as defined in~\cite{GideaZgliczynski}.
Informally speaking, one introduces compact sets called \emph{h-sets} on the Poincar\'e sections near the points $\Lambda_d(w_{*}),\Lambda_u(w_{*})$, $\Lambda_u(w^{*}),\Lambda_d(w^{*})$.
They come equipped with a coordinate system with one direction specified as exit, the other as entry.
To verify a covering relation by a Poincar\'e map between such two sets $X$, $Y$
one needs to check that the exit direction edges of $X$ are stretched over $Y$ in the exit direction
and that the image of $X$ is contained in the entry direction width of $Y$, see Figure~\ref{coveringFig} in Section~\ref{sec:covseg}.
The ``shooting'' in the exit direction is in fact made possible by a non-degenerate intersection of stable and unstable manifolds of the respective fixed
points in the singular limit $\epsilon=0$, as described in condition (P2) in Section~\ref{sec:covslowfast}.

Around the slow manifold branches we place isolating segments, which allow us to track the orbit in this region
without rigorous integration, see Theorem~\ref{is:1}.
In short, isolating segments are three-dimensional solids diffeomorphic to cubes and akin to isolating blocks of Conley and Easton~\cite{Easton}
or periodic isolating segments of Srzednicki~\cite{Srzednicki, SrzednickiWojcik}.
For each isolating segment one distinguishes three directions: exit, entry and central. 
It is required that the faces in the exit direction are immediate exit sets for the flow, the faces in the entry direction are immediate entrance sets
and the flow is monotone along the one-dimensional central direction.
The first two assumptions are checked by a computer (for $\epsilon \in [0,1.5 \times 10^{-4}]$), exploiting the hyperbolicity of branches of the slow manifold.
The last one we can easily fulfill by aligning the central direction of segments with the slow variable direction.
This setup reduces the central direction condition to a verification of whether $\frac{dw}{dt} \neq 0$ for all points in such segment. 
Under assumption $\epsilon \neq 0$ we can then factor out $\epsilon$ from the slow velocities,
and our condition reduces to a question whether $u \neq w$ for all points in each segment, which is straightforward to check.
This is the only moment in the proof, when we need to make use of the assumption that $\epsilon$ is strictly greater than $0$ (i.e. $\epsilon \in (0,1.5 \times 10^{-4}]$).

We place four supplementary ``corner segments'', containing the corner points $\Lambda_d(w_{*})$, $\Lambda_u(w_{*})$, $\Lambda_u(w^{*})$, $\Lambda_d(w^{*})$,
the role of which is to connect the h-sets with the segments around the slow manifold. 
From the viewpoint of definition these are no different than regular isolating segments.
However, the mechanism of topological tracking of the periodic orbit here is slightly distinct, as
the central direction changes roles with the exit/entry ones, see Theorems~\ref{is:3}, \ref{is:4}.

To obtain a closed loop, the sizes of the first and the last h-set in the sequence have to match.
For that purpose isolating segments around the slow manifold may need to
grow in the exit direction and compress in the entry one as we move along the orbit.
This way we can offset the size adjustments of the h-sets, which may be necessary to obtain covering relations in the fast regime.
The analysis of a model example performed in Theorem~\ref{thm:heuristic} is devoted to providing an argument for why this should work for $\epsilon$ small.
The main idea is that as $\epsilon \to 0^{+}$ the vector field in the slow/central direction decreases to 0 and
the slope of the segment becomes irrelevant when checking isolation, see Figure~\ref{slopeFig} in Section~\ref{sec:covslowfast}.

A schematic drawing representing the idea of the proof for the model example is given in Figure~\ref{schemeFig}, in Section~\ref{sec:covslowfast}.

\subsection{Outlines of the proofs of Theorems~\ref{thm:main2}, \ref{thm:main3}}

We are already at some distance from $\epsilon=0$, but for small $\epsilon$ the periodic orbit's normal bundle is consisting
of one strongly repelling and one strongly contracting direction, so any attempts of approximating the orbit by numerical integration, either forward or backward in time, fail.
On the other hand, the singular orbit at $\epsilon=0$ no longer serves as a good approximation for the purpose of a computer assisted proof.
As we can see, the challenge now is on the numerical, rather than the conceptual side.
To find our good numerical guess we introduce a large amount of sections, so that the integration times between each two of them
do not exceed some given bound and then apply Newton's method to a problem of the form
\begin{equation}\label{eq:problemform}
  \begin{aligned}
    P_{1}(x_{1}) - x_{2} &= 0, \\ P_{2}(x_{2}) - x_{3}&=0, \\ \dots  \\ P_{k}(x_{k}) - x_{1}&=0,
  \end{aligned}
\end{equation}
where $P_{i}$'s are the respective Poincar\'e maps.
Then, we construct a closed sequence of h-sets on these sections and verify covering relations between each two consecutive ones,
to prove the periodic orbit by means of Theorem~\ref{cov:sequence} (Corollary 7 in \cite{GideaZgliczynski}).
Since we control the integration times, isolating segments
are not needed anymore -- in Theorem~\ref{thm:main1} they were used for pieces of the orbit where the integration time tended to infinity.

By a rigorous continuation with parameter $\epsilon$, we are able to
get an increase of one order of magnitude for the upper bound of the range of $\epsilon$'s, for which the periodic orbit is confirmed.
Without much effort we show that for this value of $\epsilon$ the classical (see \cite{Galias, Zgliczynski, Alefeld, Neumaier} and references given there) method 
of application of the interval Newton-Moore operator to a problem of form~\eqref{eq:problemform} succeeds.
This requires a rigorous $C^{1}$ computation, but these are handled efficiently by the $C^{1}$ Lohner algorithm implemented in CAPD~\cite{Zgliczynski}.

\subsection{Organization of the paper}

The content of this paper is arranged as follows. In Section~\ref{sec:covseg}
we provide a self-contained theory on how to incorporate isolating segments into the method of covering relations.
Typical application to find periodic orbits is given by Theorem~\ref{thm:1} which works for any number of exit and entry dimensions.
This theorem is however only for future references; for the FitzHugh-Nagumo equation we use Theorem~\ref{thm:2}, which
is restricted to one exit and one entry direction and employs a certain switch between slow and fast directions in segments (Theorems~\ref{is:3}, \ref{is:4}),
which can be viewed as a topological version of the Exchange Lemma (cf. Chapter 5 in~\cite{JonesBook} and references given therein) from geometric singular perturbation theory (GSPT). 
In Section~\ref{sec:covslowfast} we give arguments
of why one can find a closed sequence of covering relations in the FitzHugh-Nagumo system for $\epsilon>0$ small: given a 3D model fast-slow system with a
singular orbit sharing the qualitative properties of the one in~\eqref{FhnOde}, we argue that it is possible to apply Theorem~\ref{thm:2} to obtain the existence of a periodic orbit.
In some sense this is a repetition of methods in~\cite{Conley,Carpenter,Hastings} recast in the language of covering relations and isolating segments.
The exposition of Section~\ref{sec:covslowfast} is informal and no part of this section is needed to prove
any theorems outside it, in particular Theorems~\ref{thm:main1}, \ref{thm:main2} or \ref{thm:main3};
the reason for including it in this paper is to give the reader some intuition for the relation between our abstract topological methods
and the fast-slow structure of the system.
Finally in Section~\ref{sec:implementation} we give details of the computer assisted proofs of Theorems~\ref{thm:main1}, \ref{thm:main2}, \ref{thm:main3}
and provide some numerical data from the programs.

\section{Covering relations and isolating segments}\label{sec:covseg}

\subsection{Notation}

Unless otherwise stated $\vectornorm{\cdot}$ can be any fixed norm in $\rr^{n}$.
Given a norm, by $B_{n}(c,r)$ we will denote the ball of radius $r$ centered at $c \in \rr^{n}$.
We will drop the subscript if the dimension is clear from the context.
By $\langle \cdot , \cdot \rangle$ we will denote the standard dot product in $\rr^{n}$.

We assume that $\rr$ is always equipped with the following norm: $\vectornorm{x} = |x|$.

Given a set $Z$, by $\inter Z$, $\overline{Z}$, $\partial Z$ and $\conv{Z}$
we will denote the interior, closure, boundary and the convex hull of $Z$, respectively.

Given a topological space $X$, a subspace $D \subset \rr \times X$ and a local flow $\varphi : D \to X$,
by writing $\varphi(t,x)$ we will implicitly state that $(t,x) \in D$,
so for example by
\begin{equation}
 \varphi(t,x) = y
\end{equation}
we will mean $\varphi(t,x)$ exists and $\varphi(t,x) = y$.

By $\id_{X}$ we denote the identity map on $X$.
The symbol $\const$ denotes a constant -- usually some uniform bound -- the value of which being not important to us,
so for example the expression $f(x) > \const, \ x \in X$, means that there exists $C \in \rr$ such that
$f(x)>C \ \forall x \in X$.

The local Brouwer degree of a continuous map $g: \rr^n \to \rr^n$ at $c \in \rr^n$ 
in an open, bounded set $\Omega \subset \rr^n$ will be denoted by $\deg(c, g, \Omega )$, provided it is well-defined.

By smoothness we mean $C^{1}$ smoothness. In some assumptions differentiability would be enough,
but we do not go into such details.

\subsection{Covering relations - basic notions}

In this subsection we recall the definitions of h-sets, covering and back-covering relations for maps as introduced in \cite{GideaZgliczynski}.
We make a following change in the nomenclature: in~\cite{GideaZgliczynski} various objects related to h-sets (directions, subsets, etc.) 
are being referred to as unstable or stable. We will refer to them as exit and entry/entrance, respectively; 
we think that this reflects better their dynamical nature and does not lead to misunderstandings.
However, we keep the original symbols $u,s$, so $u$ should be connoted with exit and $s$ with entry.

\begin{defn}[Definition 1 in~\cite{GideaZgliczynski}]\label{h-set}
  An h-set is formed by a quadruple 
  \begin{equation}
    X=(|X|,u(X),s(X),c_{X})
  \end{equation}
  consisting of
  a compact set $|X| \subset \rr^{n}$ - \emph{the support}, a pair of numbers $u(X), s(X) \in \mathbb{N}$
  such that $u(X)+s(X) = n$ (the number of exit and entry directions, respectively)
  and a coordinate change homeomorphism $c_{X}: \rr^{n} \rightarrow \rr^{u(X)} \times \rr^{s(X)}$
  such that
  \begin{equation}\label{h-set:1}
    c_X(|X|)=\overline{B_{u(X)}(0,1)} \times \overline{B_{s(X)}(0,1)}.
  \end{equation}

  We set:
  \begin{equation}
    \begin{aligned}
      X_{c}&:=\overline{B_{u(X)}(0,1)} \times \overline{B_{s(X)}(0,1)}, \\
      X_{c}^{-}&:=\partial B_{u(X)}(0,1) \times \overline{B_{s(X)}(0,1)}, \\
      X_{c}^{+}&:=\overline{B_{u(X)}(0,1)} \times \partial B_{s(X)}(0,1), \\
      X^{-} &:= c_{X}^{-1}(X_{c}^{-}), \\
      X^{+} &:= c_{X}^{-1}(X_{c}^{+}).
    \end{aligned}
  \end{equation}

  We will refer to $X^{-}$/$X^{+}$ as the exit/entrance sets, respectively.
  To shorten the notation we will sometimes drop the bars in the symbol $|X|$ and just write $X$ to denote both the
  h-set and its support.
\end{defn}

\begin{rem}
  Due to condition~\eqref{h-set:1}, it is enough to specify $u(X), s(X)$ and $c_{X}$ to unambiguously define an h-set $X$.
\end{rem}

\begin{defn}[Definitions 2, 6 in~\cite{GideaZgliczynski}]\label{covering}
  Assume that $X,Y \subset \rr^{n}$ are h-sets, such that $u(X)=u(Y)=u$ and $s(X)=s(Y)=s$. Let $g:\Omega \to \rr^{n}$ be a
  map with $|X| \subset \Omega \subset \rr^n$. Let $g_{c}= c_{Y} \circ g \circ c_{X}^{-1}: X_{c} \to \rr^{u} \times \rr^{s}$
  and let $w$ be a non-zero integer.
  We say that $X$ $g$-covers $Y$ with degree $w$ and write
  \begin{equation}
     X \cover{g,w} Y
  \end{equation}
  iff $g$ is continuous and the following conditions hold
  \begin{itemize}
    \item[1.] there exists a continuous homotopy $h: [0,1] \times X_{c} \to \rr^{u} \times \rr^{s}$, such that
     \begin{align}
       h_{0}&=g_{c},\label{cover:1a} \\
       h([0,1],X_{c}^{-}) \cap Y_{c} &= \emptyset ,\label{cover:1b}  \\
       h([0,1],X_{c}) \cap Y_{c}^{+} &= \emptyset .\label{cover:1c}
    \end{align}
  \item[2.] There exists a continuous map $A:\rr^{u} \to \rr^{u}$, such that
    \begin{equation}\label{cover:2}
      \begin{aligned}
        h_{1}(p,q)&=(A(p),0) \quad \forall p \in \overline{B_{u}(0,1)} \text{ and } q \in \overline{B_{s}(0,1)}, \\
        A(\partial B_{u}(0,1)) &\subset  \rr^{u} \setminus \overline{B_{u}(0,1)}, \\
        \deg \left( 0,A, B_u(0,1) \right) &= w.
      \end{aligned}
   \end{equation}
  \end{itemize}

  In case $A$ is a linear map, from (A8) we get $\deg \left( 0, A , B_u(0,1) \right) = \sgn \det A = \pm 1$.
   In such situation we will often say that $X$ $g$-covers $Y$, omit the degree and write $X \cover{g} Y$.
\end{defn}

\begin{rem}[Remark 3 in~\cite{GideaZgliczynski}]
  For $u=0$ we have $B_{u}(0,1) = \emptyset$ and $X$ $g$-covers $Y$ iff $g(|X|)$ is a subset of $\inter |Y|$. 
  In that case, we formally set the degree $w$ to $1$.
\end{rem}

\begin{defn}[Definition 3 in~\cite{GideaZgliczynski}]
  Let $X$ be an h-set. We define the \textit{transposed h-set} $X^{T}$ as follows:
  \begin{itemize}
    \item $|X| = |X^{T}|$,
    \item $u(X^{T}) = s(X)$ and $s(X^{T}) = u(X)$,
    \item $c_{X^{T}}(x) = j(c_{X}(x))$, where $j: \rr^{u(X)} \times \rr^{s(X)} \to \rr^{s(X)} \times \rr^{u(X)}$ is given by $j(p,q)=(q,p)$.
  \end{itemize}
\end{defn}

Observe that $(X^{T})^{+} = X^{-}$ and $(X^{T})^{-} = X^{+}$, thus transposition
changes the roles of exit and entry directions.

\begin{defn}[Definition 4, 7 in~\cite{GideaZgliczynski}]
  Let $X,Y$ be h-sets with $u(X) = u(Y)$ and $s(X) = s(Y)$. Let $g: \Omega \subset \rr^{n} \to \rr^{n}$.
  We say that $X$ $g$-backcovers $Y$ with degree $w$
  and write $X \backcover{g,w} Y$ iff 
  \begin{itemize}
    \item $g^{-1}: |Y| \to \rr^{n}$ exists and is continuous, 
    \item $Y^{T}$ $g^{-1}$-covers $X^{T}$ with degree $w$.
   \end{itemize}
\end{defn}

\begin{defn}[Definition 5 in~\cite{GideaZgliczynski}]
  We will use the notation $X \gencover{g,w} Y$ and say that $X$ generically $g$-covers $Y$ with degree $w$ iff any of these two hold:
  \begin{itemize}
    \item $X$ $g$-covers $Y$ with degree $w$,
    \item $X$ $g$-backcovers $Y$ with degree $w$.
  \end{itemize}
\end{defn}

Again, we will sometimes omit the degree in our notation, in case the homotopy can be given to a linear map.

Let us state the fundamental lemma motivating the use of covering relations for finding trajectories and fixed points of sequences of maps.

\begin{thm}[Theorem 9 in \cite{GideaZgliczynski}]\label{cov:sequence}
  Let $X_i, \ i \in \{0, \dots, k\}$ be h-sets with $u(X_0) = \dots = u(X_k)$, $s(X_0)= \dots = s(X_k)$ and set $n=u(X_0)+s(X_0)$.
  Assume that we have the following chain of covering relations:
 \begin{equation}
  X_{0} \longgencover{g_{1},w_1} X_{1} \longgencover{g_{2},w_2} X_{2} \longgencover{g_{3},w_3}  \dots \longgencover{g_{k},w_k} X_{k}.
  \end{equation}
  for some $w_i \in \zz^{*}$. Then there exists a point $x \in \inter X_{0}$ such that
 \begin{equation}
  (g_{i} \circ g_{i-1} \circ \dots \circ g_{1}) (x) \in \inter X_{i}, \ i \in \{ 1, \dots k \}.
 \end{equation}
 Moreover, if $X_{k}=X_{0}$, then $x$ can be chosen so that
 \begin{equation}
  (g_{k} \circ g_{k-1} \circ \dots \circ g_{1}) (x) = x.
 \end{equation}
\end{thm}

For an h-set $X$ with $u(X)=1, \ s(X)=s$ we have:
\begin{equation}
  \begin{aligned}
    X_{c} &= [-1,1] \times \overline{B_{s}(0,1)}, \\
    X_{c}^{-} &= (\{-1\} \times \overline{B_{s}(0,1)}) \cup (\{1\} \times \overline{B_{s}(0,1)}),
  \end{aligned}
\end{equation}

We will often use the following geometrical criterion for verifying of the covering relation in such a case:

\begin{lem}[Theorem 16 in \cite{GideaZgliczynski}]\label{covlemma}
  Let $X,Y$ be h-sets with $u(X)=u(Y)=1$, $s(X)=s(Y)=s$. Let $g:X \to \rr^{s+1}$ be a continuous map.
  Assume that both of the following conditions hold:
  \begin{itemize}
    \item[(C1)] We have
      \begin{equation}
        g_{c}(X_{c}) \subset \inter( ((-\infty, -1) \times \rr^{s}) \cup Y_{c} \cup ((1, \infty) \times \rr^{s}) ),
      \end{equation}
    \item[(C2)] either
      \begin{equation}
        \begin{aligned}
          g_{c}(\{-1\} \times \overline{B_{s}(0,1)}) \subset (-\infty, -1) \times \rr^{s} &\text{ and } g_{c}(\{1\} \times \overline{B_{s}(0,1)}) \subset (1, \infty) \times \rr^{s} \\
          &\text{or} \\
          g_{c}(\{-1\} \times \overline{B_{s}(0,1)}) \subset (1,\infty) \times \rr^{s} &\text{ and } g_{c}(\{1\} \times \overline{B_{s}(0,1)} ) \subset (-\infty, -1) \times \rr^{s}.
        \end{aligned}
      \end{equation}
  \end{itemize}

  Then
  \begin{equation}
   X \cover{g} Y.
  \end{equation}
\end{lem}

\begin{rem}\label{rem:covlemma}
In applications it is convenient to introduce the notation $X^{-,l} = c_{X}^{-1}(\{-1\} \times \overline{B_{s}(0,1)})$ (\emph{the left exit edge})
and $X^{-,r} = c_{X}^{-1}(\{1\} \times \overline{B_{s}(0,1)})$ (\emph{the right exit edge}) and check (C1), (C2)
by putting
\begin{equation}
  \begin{aligned}
    g_{c}(X_{c}) &= (c_{Y} \circ g)(|X|),\\
    g_{c}(\{-1\} \times \overline{B_{s}(0,1)}) &= (c_{Y} \circ g) ( X^{-,l} ),\\
    g_{c}(\{1\} \times \overline{B_{s}(0,1)}) &= (c_{Y} \circ g) ( X^{-,r} ),
  \end{aligned}
\end{equation}
see Figure~\ref{coveringFig}.
\end{rem}

\begin{figure}
    \centering
      \begin{tikzpicture}[line cap=round,line join=round,>=latex,x=2.5mm,y=2.5mm]
 
  \clip(-19,-8.4) rectangle (19,7);

  \fill [color=brown!50] (-5.,-5.) -- (-5.,5.) -- (5.,5.) -- (5.,-5.);
  \draw [color=orange,thick] (-5.,5.) -- (-5.,-5.);
  \draw [color=orange,thick] (5.,5.) -- (5.,-5.);
  \draw [color=brown,thick] (-5.,5.) -- (5.,5.);
  \draw [color=brown,thick] (5.,-5.) -- (-5.,-5.);

  \filldraw[draw=blue, fill=blue!50, opacity=0.8]
      (-9,-4) .. controls (-2,1) and (3,-4) .. (11,-1) -- (10,3) .. controls (3,0) and (-2,5) .. (-10,0) -- (-9,-4);

  \draw [color=red,thick] (11,-1) -- (10,3);
  \draw [color=red,thick] (-10,0) -- (-9,-4);

  \draw (3.,-5.) node[anchor=south] {\scriptsize{$\displaystyle Y_{c}$}};
  \draw (0.,1.) node[anchor=north] {\scriptsize{$\displaystyle(c_{Y} \circ g)\left(|X|\right)$}};
  \draw (10.5,1.) node[anchor=west] {\scriptsize{$\displaystyle(c_{Y} \circ g)(X^{-,r})$}};
  \draw (-9.5,-2.) node[anchor=east] {\scriptsize{$\displaystyle(c_{Y} \circ g)(X^{-,l})$}};

  \draw[color=black,->] (-7.,-7.) -- (-5,-7.);
  \draw[color=black,->] (-7.,-7.) -- (-7,-5.);
  \draw (-7.,-6) node[anchor=east] {\scriptsize{$x_{s}$}};
  \draw (-6.,-7) node[anchor=north] {\scriptsize{$x_{u}$}};

\end{tikzpicture}
      \caption{A covering relation $X \cover{g} Y$.}
      \label{coveringFig}
\end{figure}
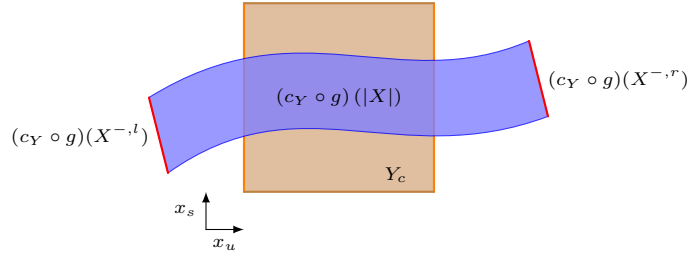

Analogously, if for an h-set $Y$ we have $u(Y)=u$ and $s(Y)=1$, then
\begin{equation}
  \begin{aligned}
    Y_{c} &= \overline{B_{u}(0,1)} \times [-1,1], \\
    Y_{c}^{+} &= (\overline{B_{u}(0,1)} \times \{-1\}) \cup (\overline{B_{u}(0,1)} \times \{1\}),
  \end{aligned}
\end{equation}
and we can apply the same principle to transposed sets:
\begin{lem}\label{backcovlemma}
  Let $X,Y$ be h-sets with $u(X)=u(Y)=u$ and $s(X)=s(Y)=1$. Let $\Omega \subset \rr^{u+1}$ and $g: \Omega \to \rr^{u+1}$ be continuous.
  Assume, that $g^{-1}: |Y| \to \rr^{u+1}$ exists, is continuous and that both of the following conditions hold:
  \begin{itemize}
    \item[(C1a)] We have
      \begin{equation}
        g^{-1}_{c}(Y_{c}) \subset \inter( ( \rr^{u} \times (-\infty, -1) ) \cup X_{c} \cup (\rr^{u} \times (1, \infty)) );
      \end{equation}
    \item[(C2a)] either
      \begin{equation}
        \begin{aligned}
          g^{-1}_{c}( \overline{B_{u}(0,1)} \times \{-1\} ) \subset \rr^{u} \times (-\infty, -1) &\text{ and }g^{-1}_{c}(\overline{B_{u}(0,1)} \times \{1\}) \subset \rr^{u} \times (1,\infty) \\
          &\text{or} \\
          g^{-1}_{c}( \overline{B_{u}(0,1)} \times \{-1\}) \subset \rr^{u} \times (1,\infty) \text{ and } &g^{-1}_{c}(\overline{B_{u}(0,1)} \times \{1\}) \subset \rr^{u} \times(-\infty, -1) .
        \end{aligned}
      \end{equation}
  \end{itemize}
  Then
  \begin{equation}
   X \backcover{g} Y.
  \end{equation}
\end{lem}

In such case we will sometimes operate with the notation $Y^{+,l} = c_{Y}^{-1}( \overline{B_{u}(0,1)} \times \{-1\} )$ (\emph{the left entrance edge})
and $Y^{+,r} =  c_{Y}^{-1}( \overline{B_{u}(0,1)} \times \{1\} )$ (\emph{the right entrance edge}).
Conditions (C1a) and (C2a) can then be rephrased in the same manner as in Remark~\ref{rem:covlemma}.

\subsection{Covering relations and Poincar{\'e} maps}\label{subsec:poinc}

In this subsection we will consider an ODE
\begin{equation}\label{ODE}
  \begin{aligned}
    \dot{x} &= f(x), \\
    \ x &\in \rr^{N}
  \end{aligned}
\end{equation}
given by a smooth vector field $f$, and describe how h-sets and covering relations can be used in such setting.

Assume, that we are given a diffeomorphism $\Phi: \rr^N \to \rr^N$, 
and let $\Sigma \subset \rr^{N}$ be a subset of the hypersurface $\Xi:=\Phi^{-1}( \{0\} \times \rr^{N-1} )$.
A point $x \in \Xi$ is \emph{regular} iff $\langle n(x), f(x) \rangle \neq 0$, where $n(x)$ is a normal to $\Xi$ at $x$.
If every point $x \in \Sigma$ is regular, then we will say that $\Sigma$ is a \emph{transversal section}.

For a given $x_{0} \in \rr^{N}$ we will denote by
\begin{equation}\label{flow}
  \varphi(t,x_{0})
\end{equation}
the local flow generated by $f$, that is the value of the solution $x(t)$
to~\eqref{ODE} with the initial condition $x(0)=x_{0}$.

Let $\Sigma_{1},\ \Sigma_{2}$ be two transversal sections such that:
\begin{itemize}
  \item either $\Sigma_1 \subset \Sigma_2$ or $\Sigma_1 \cap \Sigma_2 = \emptyset$,
  \item we have $\overline{\inter \Sigma_1} = \Sigma_1$ and $\overline{\inter \Sigma_2} = \Sigma_2$, where closures and interiors are taken in
    the hypersurface topology,
  \item for each $x \in \Sigma_1$ there exists a $\tau > 0$ such that $\varphi(\tau,x) \in \Sigma_2$
\end{itemize}
It is well known that the \emph{Poincar\'e map}:
\begin{equation}\label{poincareMap}
  P : \Sigma_1 \ni x \to  \inf_{\tau: \varphi(\tau,x) \in \Sigma_{2}} \varphi(\tau,x) \in \Sigma_{2}
\end{equation}
is well-defined and smooth for points $x \in \inter \Sigma_1$ such that $P(x) \in \inter \Sigma_2$ (interiors in the hypersurface topology).
The proof can be found in e.g.~\cite{Kapela}.

To make the formulation of some theorems in future shorter, we extend the above definition of a Poincar\'e map
in the scenario $\Sigma_1 \subset \Sigma_2$ to also cover the embedding by identity $\id : \Sigma_1 \to \Sigma_2$.
In such case we will always specifically refer to such map as the identity map, to differentiate from a Poincar\'e map $P$ given by~\eqref{poincareMap}.

For such a Poincar\'e map we define the h-sets in a natural manner. We can identify $\Sigma_{1}, \Sigma_{2}$ with two copies
of $\rr^{N-1}$. Then we can proceed to describe the h-sets on each of these copies - note that they will be h-sets
in $\rr^{N-1}$, not $\rr^{N}$.

\begin{rem}\label{rem:hset}
  Treating h-sets as subsets of sections is a slight abuse when compared to Definition~\ref{h-set}, where they were subsets of 
  the Euclidean space $\rr^{N}$.
  Nevertheless, we can always compose the change of coordinates homeomorphism for the h-set with the global coordinate frame on the section
  to get back to the Euclidean space.
  Therefore, given a section $\Sigma \subset \rr^{N}$, for an h-set $X \subset \Sigma$ the actual coordinate
  change will take the form $c_{X} =  \tilde{c}_{X} \circ \Phi$,
  where $\Phi : \Sigma \to \{0\} \times \rr^{N-1}$ is the global coordinate frame for the section and $\tilde{c}_{X} : \rr^{N-1} \to \rr^{u(X)} \times \rr^{s(X)}$
  is a coordinate change homeomorphism satisfying \eqref{h-set:1}.
\end{rem}

\subsection{Isolating segments}\label{subsec:segments}

Assume we are given a smooth vector field~\eqref{ODE}, an associated local flow~\eqref{flow}
and a pair of transversal sections $\Sigma_{\text{in}}, \ \Sigma_{\text{out}}$.

\begin{defn}
  A segment between two transversal sections $\Sigma_{\text{in}}$ and $\Sigma_{\text{out}}$ is formed by a quadruple $S = (|S|, u(S), s(S), c_{S})$,
  consisting of a compact set $|S| \subset \rr^{N}$ (\emph{the support}), a pair of numbers $u(S), s(S) \in \mathbb{N}$
  with $u(S) + s(S) = N-1$ (the number of exit and entrance directions, respectively)
  and a coordinate change diffeomorphism $c_{S}: \rr^{N} \to \rr^{u(S)} \times \rr^{s(S)} \times \rr$ such that:
  \begin{equation}
    \begin{aligned}
      c_{S}(|S|) &= \overline{B_{u(S)}(0,1)} \times \overline{B_{s(S)}(0,1)} \times [0,1], \\
      c_{S}^{-1}( \overline{B_{u(S)}(0,1)} \times \overline{B_{s(S)}(0,1)} \times \{0\} ) &\subset \Sigma_{\text{in}}, \\
      c_{S}^{-1}( \overline{B_{u(S)}(0,1)} \times \overline{B_{s(S)}(0,1)} \times \{1\} ) &\subset \Sigma_{\text{out}} . \\
    \end{aligned}
  \end{equation}

  We set:
  \begin{equation}
    \begin{aligned}
      S_{c}&:=\overline{B_{u(S)}(0,1)} \times \overline{B_{s(S)}(0,1)} \times [0,1], \\
      S_{c}^{-}&:=\partial B_{u(S)}(0,1) \times \overline{B_{s(S)}(0,1)} \times [0,1], \\
      S_{c}^{+}&:=\overline{B_{u(S)}(0,1)} \times \partial B_{s(S)}(0,1) \times [0,1], \\
      S^{-} &:= c_{S}^{-1}(S_{c}^{-}), \\
      S^{+} &:= c_{S}^{-1}(S_{c}^{+}).
    \end{aligned}
  \end{equation}

  We will refer to $S^{-}$/$S^{+}$ as \emph{the exit/entrance sets}, respectively.
  Again, to shorten the notation sometimes we will drop the bars in the symbol $|S|$ and just write $S$ to denote both the
  segment and its support.
\end{defn}

\begin{rem}
  As with h-sets, it is enough to give $u(S)$, $s(S)$ and $c_{S}$ to define a segment $S$.
\end{rem}

Given a segment $S$ we introduce the following notation for projections:

\begin{equation}
  \begin{aligned}
    \pi_{u}: \rr^{u(S)} \times \rr^{s(S)} \times \rr \ni (x_{u},x_{s},x_{\mu}) &\to x_{u} \in \rr^{u(S)}, \\
    \pi_{s}: \rr^{u(S)} \times \rr^{s(S)} \times \rr \ni (x_{u},x_{s},x_{\mu}) &\to x_{s} \in \rr^{s(S)}, \\
    \pi_{\mu}: \rr^{u(S)} \times \rr^{s(S)} \times \rr \ni (x_{u},x_{s},x_{\mu}) &\to x_{\mu} \in \rr.
  \end{aligned}
\end{equation}

\begin{defn}\label{defn:isegment}
  We say that $S$ is an isolating segment between two transversal sections $\Sigma_{\text{in}}$ and $\Sigma_{\text{out}}$
  if $S$ is a segment,
  the functions $x \to \vectornorm{\pi_{u}(x)}$, $x \to \vectornorm{\pi_{s}(x)}$, $x \to \vectornorm{\pi_{\mu}(x)}$,
  $x \in \rr^{u(S)} \times \rr^{s(S)} \times \rr$ are smooth everywhere except at $0$
  and the following conditions are satisfied:
  \begin{itemize}
    \item[(S1)] $\frac{d}{dt} \pi_{\mu} c_{S}( \varphi(t,x) )_{|_{t=0}} > 0$ for all $x \in |S|$ (\emph{monotonicity}),
    \item[(S2)] $\frac{d}{dt} \vectornorm{ \pi_{u} c_{S}( \varphi(t,x) ) }_{|_{t=0}} > 0$ for all $x \in S^{-}$ (\emph{exit set isolation}),
    \item[(S3)] $\frac{d}{dt} \vectornorm{ \pi_{s} c_{S}( \varphi(t,x) ) }_{|_{t=0}} < 0$ for all $x \in S^{+}$ (\emph{entrance set isolation}).
  \end{itemize}
\end{defn}

As one can see, our definition of an isolating segment $S$ relies on splitting the phase space into:
\begin{itemize}
  \item the \emph{exit} directions $\pi_{u} \circ c_{S}$,
  \item the \emph{entry} directions $\pi_{s} \circ c_{S}$,
  \item the one-dimensional \emph{central} direction $\pi_{\mu} \circ c_{S}$.
\end{itemize}
In that sense, it is a simplification of the concept of periodic isolating segments in nonautonomous systems, as originally introduced in~\cite{SrzednickiWojcik}
(also, under a name of periodic isolating blocks in~\cite{Srzednicki}),
where a wider range of boundary behavior was considered.
On the other hand, contrary to~\cite{SrzednickiWojcik}, we are able to work with an autonomous ODE
-- in~\cite{SrzednickiWojcik} the central direction had to be given by time.

When introducing an isolating segment we will sometimes omit specifying the transversal sections, as they are implicitly defined by $c_{S}$.

\begin{rem}\label{alternativeS}
  Each of the conditions (S1)-(S3) is equivalent to its following counterpart:
  \begin{itemize}
    \item[(S1a)] $\langle \nabla (\pi_{\mu} \circ c_{S})(x), f(x) \rangle > 0$ for all $x \in |S|$,
    \item[(S2a)] $\langle \nabla \vectornorm{ \pi_{u} \circ c_{S} }(x), f(x) \rangle > 0$ for all $x \in S^{-}$,
    \item[(S3a)] $\langle \nabla \vectornorm{ \pi_{s} \circ c_{S} }(x), f(x) \rangle < 0$ for all $x \in S^{+}$.
  \end{itemize}

  Since $S^{-}$, $S^{+}$ are subsets of the level sets $\{ x \in \rr^{N}: \vectornorm{ \pi_{u} \circ c_{S}(x) } = 1 \}$,
  $\{ x \in \rr^{N}: \vectornorm{ \pi_{s} \circ c_{S}(x) } = 1 \}$, respectively, and gradients are normals to level sets,
  (S2a) and (S3a) can also be rewritten as:
  \begin{itemize}
    \item[(S2b)] $\langle n_{-}(x), f(x) \rangle > \const > 0$ for all $x \in S^{-}$,
    \item[(S3b)] $\langle n_{+}(x), f(x) \rangle < \const < 0$ for all $x \in S^{+}$,
  \end{itemize}
  respectively, 
  where $n_{\mp}(x)$ are normals to $S^{\mp}$, pointing in the outward direction of $|S|$\footnote{The sets $S^{\mp}$ are manifolds with boundary, so by normals
  at the boundary points we mean normals to any smooth extension of $S^{\mp}$ to a manifold without boundary.}.
\end{rem}

In our applications the faces of segments will always lie in affine subspaces,
hence conditions (S2b) and (S3b) are easy to check by an explicit computation.
Let $\pi_i$ be the projection onto $i$-th variable, $i \in 1,\dots,N$.
In the central direction our changes of coordinates will take an affine form
\begin{equation}\label{eq:centralform}
  \pi_{\mu} c_{S}(x) = a \pi_i (x) + b,
\end{equation}
for $a \neq 0,\ b \in \rr$.
In that situation (S1a) is equivalent with
\begin{equation}\label{eq:sgna}
  \sgn(a) \pi_i (f(x)) > 0, \ \forall x \in |S|,
\end{equation}
which again is easily established. In particular, if the sign of $\pi_i(f(x))$ is negative,
one needs to orient the segment in the direction reverse to the $i$-th coordinate direction
by giving $a$ a negative sign.

We will now introduce the notion of the transposed segment, analogous to the transposed h-set.

\begin{defn}
  Given a segment $S$ between two transversal sections $\Sigma_{\text{in}}$ and $\Sigma_{\text{out}}$ we define the
  transposed segment $S^{T}$ between $\Sigma_{\text{out}}$ and $\Sigma_{\text{in}}$ by setting:
  \begin{equation}
    \begin{aligned}
      |S^{T}| &:= |S|, \\
      u(S^{T}) &:= s(S), \\
      s(S^{T}) &:= u(S), \\
      c_{S^{T}} &:= {\rm o} \circ c_{S};
    \end{aligned}
  \end{equation}
  where ${\rm o}: \rr^{u(S)} \times \rr^{s(S)} \times \rr \to \rr^{u(S^{T})} \times \rr^{s(S^{T})} \times \rr$, ${\rm o}(p,q,r) = (q,p,1-r)$.

  Observe that
  \begin{equation}
    \begin{aligned}
      (S^{T})^{-} &= S^{+}, \\
      (S^{T})^{+} &= S^{-}.
    \end{aligned}
  \end{equation}

\end{defn}

\begin{prop}\label{ts:1}
  Let $S$ be an isolating segment between transversal sections $\Sigma_{\text{in}}$ and $\Sigma_{\text{out}}$ for $\dot{x} = f(x)$.
  Then $S^{T}$ is an isolating segment between $\Sigma_{\text{out}}$ and $\Sigma_{\text{in}}$ for the inverted vector field $\dot{x} = -f(x)$.
  The sections $\Sigma_{\text{out}}$ and $\Sigma_{\text{in}}$ are transversal for the inverted vector field.
\end{prop}

\subsection{Isolating segments imply coverings}

Given a segment $S$ between transversal sections $\Sigma_{\text{in}}$ and $\Sigma_{\text{out}}$ there is a natural structure of h-sets defined
on the faces given by intersections $\Sigma_{\text{in}} \cap |S|$ and $\Sigma_{\text{out}} \cap |S|$.

\begin{defn}
  We define the h-sets:
  \begin{itemize}
    \item $X_{S,\text{in}} \subset \Sigma_{\text{in}}$ (\emph{the front face}),
    \item $X_{S,\text{out}} \subset \Sigma_{\text{out}}$ (\emph{the rear face}),
  \end{itemize}
  as follows:
  \begin{itemize}
    \item $u( X_{S,\text{in}} ) = u( X_{S,\text{out}} ) := u(S)$ and $s( X_{S,\text{in}} ) = s( X_{S,\text{out}} ) := s(S)$;
    \item $|X_{S, \text{in}}| =\Sigma_{\text{in}} \cap |S|$ and $|X_{S, \text{out}}| = \Sigma_{\text{out}} \cap |S|$;
    \item $c_{ X_{S,\text{in}} } := (\pi_{u},\pi_{s}) \circ {c_{S}}_{|_{\Sigma_{\text{in}}}}$ 
      and $c_{ X_{S,\text{out}} } := (\pi_{u},\pi_{s}) \circ {c_{S}}_{|_{\Sigma_{\text{out}}}}$.
  \end{itemize}

\end{defn}

\begin{defn}
  Let $S$ be an isolating segment. We define \emph{the exit map} $E_{S}: |X_{S,\text{in}}| \to S^{-} \cup |X_{S,\text{out}}|$ by
  \begin{equation}
   E_{S}(x)=\varphi(t_{e},x), \ t_{e} = \min \left\{ t \in \rr^{+} \cup \{0\}: \ \varphi(t,x) \in S^{-} \cup |X_{S,\text{out}}| \right\},
  \end{equation}
  and \emph{the persistent set} by
  \begin{equation}
   S^{0} := \{ x \in |X_{S,\text{in}}|: E_{S}(x) \in |X_{S,\text{out}}| \}.
  \end{equation}
\end{defn}

\begin{rem}
  From (S1), (S2), (S3) it follows that the function $E_{S}$ is well-defined and a homeomorphism onto its image. 
\end{rem}

\begin{thm}\label{is:1}
  Let $S$ be an isolating segment between transversal sections $\Sigma_{\text{in}}$ and $\Sigma_{\text{out}}$.
  Define $V := \{x \in \Sigma_{\text{in}}: \exists \tau > 0 : \varphi(\tau,x) \in \Sigma_{\text{out}} \}$
  and a Poincar\'e map $P: V \to \Sigma_{out}$ as in equation~\eqref{poincareMap}. Then
  \begin{itemize}
    \item $V \neq \emptyset$;
    \item there exists a diffeomorphism $R: \Sigma_{\text{in}} \to \Sigma_{\text{out}}$ such that we have a covering relation
      \begin{equation}
        X_{S, \text{in}} \cover{R} X_{S, \text{out}}
      \end{equation}
    and
    \begin{equation}\label{R=P}
      P(x) = R(x) \quad \forall x \in S^{0};
    \end{equation}
    \item it holds that 
      \begin{equation}\label{S0=RSet}
        S^{0}= \{ x \in |X_{S,\text{in}}|: R(x) \in |X_{S,\text{out}}| \}.
      \end{equation}
      In particular, for every $x \in |X_{S, \text{in}}|$ such that $R(x) \in |X_{S, \text{out}}|$ the part of the trajectory
      between $x$ and $P(x)=R(x)$ is contained in $|S|$.
  \end{itemize}
\end{thm}

The intuition behind this theorem is portrayed in Figure~\ref{segmentCovFig}.

\begin{figure}
    \centering
      \input{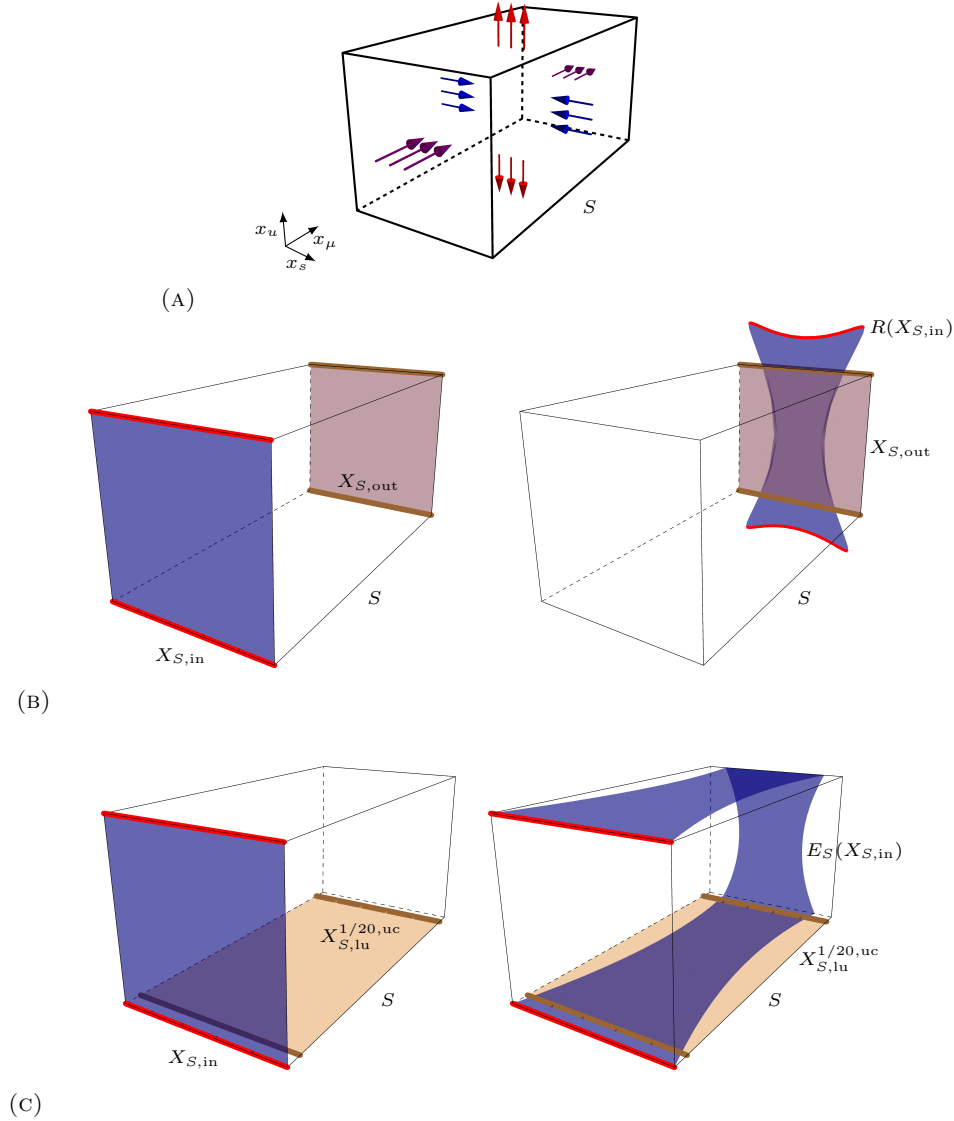}
      \caption{Schematic drawings of covering relations in a three-dimensional isolating segment $S$ with $u(S)=s(S)=1$.}
      \label{segmentCovFig}
\end{figure}

\begin{proof}
  To make the formulas clearer, without any loss of generality we assume that $c_{S}=\id_{\rr^{N}}$. Define
  \begin{equation}
    g(x) = \frac{d}{dt} \tilde{\varphi}( t, x )_{|_{t=0}},
  \end{equation}
  where
  \begin{equation}
    \tilde{\varphi}(t,x) = \left( \begin{array}{c} e^{t} \pi_{u}(x) \\ e^{-t} \pi_{s}(x) \\ t + \pi_{\mu}(x) \end{array} \right)
  \end{equation}
  is a global flow on $\rr^{u(S)} \times \rr^{s(S)} \times \rr$.

  We have
    \begin{align}
      \langle \nabla \vectornorm{ \pi_{u} (x) }, g(x) \rangle &= \frac{d}{dt} \vectornorm{ \pi_{u} \tilde{\varphi}(t, x ) }_{|_{t=0}} = \vectornorm{ \pi_{u} (x) },\label{eq:ga} \\
      \langle \nabla \vectornorm{ \pi_{s} (x) }, g(x) \rangle &= \frac{d}{dt} \vectornorm{ \pi_{s} \tilde{\varphi}(t, x ) }_{|_{t=0}} = -\vectornorm{ \pi_{s} (x) },\label{eq:gb}
    \end{align}
  for all $x \neq 0$ and   
  \begin{equation}
    \langle \nabla \pi_{\mu} (x), g(x) \rangle = \frac{d}{dt} \pi_{\mu} \tilde{\varphi}(t, x )_{|_{t=0}} = 1,\label{eq:gc}
  \end{equation}
  for all $x \in \rr^{N}$.

  Let $U$ be a bounded, open neighborhood of $|S|$, small enough so that the following conditions are satisfied:
    \begin{align}
      \langle \nabla \vectornorm{ \pi_{u} (x) }, f(x) \rangle > \const > 0 \ &\forall x \in U: \vectornorm{ \pi_{u} (x) } \geq 1, \label{eq:fa}\\
      \langle \nabla \vectornorm{ \pi_{s} (x) }, f(x) \rangle < \const < 0 \ &\forall x \in U: \vectornorm{ \pi_{s} (x) } \geq 1, \label{eq:fb}\\
      \langle \nabla \pi_{\mu} (x), f(x) \rangle > \const > 0 \ &\forall x \in U.\label{eq:fc}
    \end{align}

  Let $\eta: \rr^{N} \to [0,1]$ be a $C^{\infty}$ function equal to $1$ on $|S|$ and equal to $0$ on $\rr^{N} \backslash U$.
  Put $\hat{f}(x) = \eta(x) f(x) + (1-\eta(x))g(x)$. Denote by $\hat{\varphi}(t,x)$ the local flow generated by
  \begin{equation}
    \dot{x} = \hat{f}(x), \quad x \in \rr^{N}.
  \end{equation}
  Since $\hat{\varphi}(t,x) = \tilde{\varphi}(t, x),\ x \in \rr^{N} \backslash U$ and $U$ is bounded,
  $\hat{\varphi}$ is also a global flow.
  From~\eqref{eq:gc} and~\eqref{eq:fc} we have:
  \begin{equation}
    \begin{aligned}
      \langle \nabla \pi_{\mu} (x), \hat{f}(x) \rangle
      &= \eta(x)\langle \nabla \pi_{\mu} (x), f(x) \rangle
      + (1 - \eta(x)) \langle \nabla \pi_{\mu} (x), g(x) \rangle
      \\ &> \const > 0
    \end{aligned}
  \end{equation}
  for all $x \in \rr^{N}$. Therefore, the Poincar{\'e} map
  \begin{equation}
    P_{\hat{f}}: \rr^{u(S)} \times \rr^{s(S)} \times \{ 0 \} \to \rr^{u(S)} \times \rr^{s(S)} \times \{ 1 \},
  \end{equation}
  is a well-defined diffeomorphism. We set $R := P_{\hat{f}}$.

  First, we will prove that $X_{S, \text{in}}$ $R$-covers $X_{S, \text{out}}$, cf. Figures~\ref{segmentCovFig:A} and~\ref{segmentCovFig:B}.
  In what is below we identify the spaces $\rr^{u(S)} \times \rr^{s(S)} \times \{0\}$ and $\rr^{u(S)} \times \rr^{s(S)} \times \{1\}$ 
  with two copies $\rr^{u(S)} \times \rr^{s(S)}$ wherever necessary, by projecting/embedding the first $u(S)+s(S)$ coordinates.
 
  We need a homotopy of $R$ to a linear map.
  Consider the parameterized family of vector fields
  \begin{equation}
    f_{\xi}(x) = (1 - \xi) \hat{f}(x) + \xi g(x), \quad x \in \rr^{N},
  \end{equation}
  where $\xi \in [0,1]$. By the same reasoning as with $\hat{f}$ each of these vector fields
  generates a global flow and induces an associated Poincar{\'e} map
  \begin{equation}
    P_{f_{\xi}}: \rr^{u(S)} \times \rr^{s(S)} \times \{ 0 \} \to \rr^{u(S)} \times \rr^{s(S)} \times \{ 1 \}.
  \end{equation}
  We define a continuous
  homotopy of maps $h: [0,1] \times \rr^{u(S)} \times \rr^{s(S)} \to \rr^{u(S)} \times \rr^{s(S)}$:
  \begin{equation}
    \begin{aligned}
      h(\xi, \cdot )&:= P_{f_{2\xi}},\quad \xi \in [0, 1/2], \\
      h(\xi, \cdot )&:= \left[ \begin{array}{cc} e\id_{\rr^{u(S)}} & 0 \\ 0 & (2-2\xi)e^{-1}\id_{\rr^{s(S)}} \end{array} \right],\quad \xi \in [1/2, 1].
    \end{aligned}
  \end{equation}
  Indeed, the homotopy agrees at $1/2$. Moreover, $h(0,\cdot) = R$
  and $h(1,\cdot)$ is a linear map satisfying the requirements given by~\eqref{cover:2}. Since 
  it is also clear that~\eqref{cover:1b} and~\eqref{cover:1c}
  hold for $\xi \in [1/2, 1]$, we proceed to check these two conditions on the other half of the interval.

  Denote by $\varphi^{\xi}$ the family of global flows generated by $\dot{x} = f_{\xi}(x)$.
  From~\eqref{eq:fa} and~\eqref{eq:ga}, for $\xi \in [0, 1]$ and $x: \vectornorm{\pi_{u} (x)} \geq 1$ we get

  \begin{equation}\label{hatf:1}
   \begin{aligned}
     &\frac{d}{dt} \vectornorm{ \pi_{u}  \varphi^{\xi}(t, x)  ) }_{|_{t=0}} \\
     & =\langle \nabla \vectornorm{ \pi_{u} (x)}, f_{\xi}(x) \rangle\\
     &= (1-\xi)\langle \nabla \vectornorm{ \pi_{u} (x)} , \hat{f}(x) \rangle
     + \xi \langle \nabla \vectornorm{\pi_{u} (x)} , g(x) \rangle  \\
     &= (1-\xi) \eta(x) \langle \nabla \vectornorm{\pi_{u} (x)} , f(x) \rangle
     + ( 1-\eta(x) + \xi \eta(x) ) \langle \nabla \vectornorm{\pi_{u} (x)} , g(x) \rangle \\
      & > \const > 0.
   \end{aligned}
  \end{equation}

  Therefore, $\vectornorm{\pi_{u} (x)} = 1$ implies $\vectornorm{\pi_{u}(P_{f_{\xi}}(x))} > 1$ for all $\xi \in [0,1]$
  and proves \eqref{cover:1b}.
  By a mirror argument, from~\eqref{eq:fb} and~\eqref{eq:gb} we obtain 
 \begin{equation}\label{hatf:2}
   \begin{aligned}
     \frac{d}{dt} \vectornorm{ \pi_{s}  \varphi^{\xi}(t, x)  ) }_{|_{t=0}} < \const < 0, \quad x: \vectornorm{\pi_{s}(x)} \geq 1,
   \end{aligned}
  \end{equation} 
  hence $\vectornorm{\pi_{s}(P_{f_{\xi}}(x))} = 1$ implies $\vectornorm{\pi_{s} (x)} > 1$ for all $\xi \in [0,1]$.
  This proves \eqref{cover:1c}.

  We are left to prove~\eqref{R=P} and~\eqref{S0=RSet}. 
  Let us start with the latter. 
  
  Observe, that $\hat{f}_{|_{|S|}} = f_{|_{|S|}}$, which proves the ``$\subset$'' inclusion.
  For the other one we proceed as follows.
  Since $\varphi^{0} = \hat{\varphi}$, from~\eqref{hatf:1} and~\eqref{hatf:2}
  we obtain the forward invariance of the sets $\{ x \in \rr^{N}: \  \vectornorm{ \pi_{u}  x } \geq 1 \}$,
  $\{ x \in \rr^{N}: \  \vectornorm{ \pi_{s} ( x ) } \leq 1 \}$ under $\hat{\varphi}$.
  Therefore, for $x \in |X_{S,\text{in}}|$ such that $R(x) \in |X_{S,\text{out}}|$ we have
  \begin{equation}
    \begin{aligned}
      \vectornorm{ \pi_{u}\varphi(t,x) } \leq 1, \ \forall t \geq 0: \pi_{\mu}  \varphi(t,x)  \leq 1, \\
      \vectornorm{ \pi_{s}\varphi(t,x) } \leq 1, \ \forall t \geq 0: \pi_{\mu}  \varphi(t,x)  \geq 0.
    \end{aligned}
  \end{equation}

  As a consequence the part of the trajectory between $x$ and $P_{\hat{f}}(x)$ is wholly contained in $|S|$, where the vector field $\hat{f}$ is equal to $f$,
  hence $x \in S^{0}$. By the same argument~\eqref{S0=RSet} implies~\eqref{R=P}.
\end{proof}

One can also prove an analogous backcovering lemma, which is superfluous in the context of our applications and therefore will be omitted. 

We are now ready to state a prototypical theorem on how a sequence of covering relations and isolating segments
forces the existence of periodic orbits for a given differential equation.
We do not use it though for the FitzHugh-Nagumo equations;
a more refined version, which includes a switch between the slow and the fast dynamics -- Theorem~\ref{thm:2} -- is applied instead. 

\begin{thm}\label{thm:1}
  Let $\dot{x}=f(x), \ x \in \rr^{N}$ be given by a smooth vector field. Assume that there exists a sequence of transversal sections $\{ \Sigma_{i} \}_{i=0}^{k},\ k \in \mathbb{N}$
  and a sequence of h-sets
  \begin{equation}
   \mathcal{X} = \{X_{i}: \ |X_{i}| \subset \Sigma_{i}, \ i = 0,\dots, k \},
  \end{equation}
  such that for each two consecutive h-sets $X_{i-1}$, $X_{i} \in \mathcal{X}$ we have one of the following:
  \begin{itemize}
    \item there exists a Poincar\'e map $P_{i} : \Omega_{i-1} \to \Sigma_{i}$ with $\Omega_{i-1} \subset \Sigma_{i-1}$ and an integer $w_i \in \zz^{*}$ such that
        \begin{equation}\label{eq:pgencover}
         X_{i-1} \longgencover{P_{i},w_i} X_{i},
        \end{equation}
      \item there exists an isolating segment $S_{i}$ between $\Sigma_{i-1}$ and $\Sigma_{i}$ such that $X_{S_{i},\text{in}} = X_{i-1}$ and $X_{S_{i},\text{out}} = X_{i}$.
  \end{itemize}
  Then, there exists a solution $x(t)$ of the differential equation passing consecutively through the interiors of all $X_{i}$'s.
  Moreover: 
  \begin{itemize}
    \item whenever $X_{i-1}$ and $X_{i}$ are connected by an isolating segment, the solution passes through $S^{0}_{i}$;
    \item if $X_{0} = X_{k}$ the solution $x(t)$ can be chosen to be periodic.
  \end{itemize}
\end{thm}

\begin{proof}
  By applying Theorem~\ref{is:1} we get a chain of covering relations
  \begin{equation}
    X_{0} \longgencover{g_{1},w_1} X_{1} \longgencover{g_{2},w_2} X_{2} \longgencover{g_{3},w_3} \dots \longgencover{g_{k},w_k} X_{k},
  \end{equation}
  where $g_{i} = P_{i}$ or $g_{i} = R_{i}$, $R_{i}$ being the diffeomorphism given by Theorem~\ref{is:1} associated with the segment $S_{i}$
  (then $w_i=\pm 1$).
  From Theorem~\ref{cov:sequence} there exists a sequence
  $\{ x_{i}: x_{i} \in \inter |X_{i}|, i = 1,\dots, k \}$ such that $g_{i}( x_{i-1} ) = x_{i}$ and we can choose $x_{0}=x_{k}$ whenever $X_{0}=X_{k}$.

  Suppose that for certain $i$'s we have $g_{i} = R_{i}$. Since $x_{i-1} \in |X_{i-1}|$ and $R_{i}(x_{i-1}) = x_{i} \in |X_{i}|$,
  Theorem~\ref{is:1} implies that $x_{i-1} \in S_{i-1}^{0}$ and 
  $R_{i}(x_{i-1}) = P_{i}(x_{i-1})$, $P_{i}:V_{i-1} \to \Sigma_{i}$ being a Poincar\'e map defined on a subset of $\Sigma_{i-1}$.
  This proves that this orbit is an orbit of a full sequence of Poincar\'e maps, hence a real trajectory for the flow.
  Furthermore, it is a periodic trajectory if $x_{0} = x_{k}$ (notice that it cannot be an equilibrium as the vector field on transversal sections
  cannot equal $0$).
\end{proof}

\begin{cor}
  For an isolating segment $S$ the set $S^{0}$ is nonempty.
\end{cor}

\subsubsection{Additional coverings within an isolating segment -- the ``fast-slow switch''.}

Let us first explain the ideas behind this subsection without formality.
Consider a three-dimensional isolating segment $S$ with one exit and one entry direction
let us write $X_{S,\text{lu}}$, $X_{S,\text{ru}}$ for the two connected components of the exit set $S^{-}$
a ``left exit'' and a ``right exit'' one, respectively.
Each of them lies within a level set given by fixing the exit direction level to $\mp 1$.
They can be equipped with an h-set structure with one exit and one entry direction by
setting the entry direction of the segment as the entry one
and the central direction of the segment as the exit one.

If we now consider the function $E_{S}$, which maps each point of the front face $X_{S,\text{in}}$
to the point of $\partial S$ where the trajectory leaves $S$, then 
its image will give a similar alignment as in Lemma~\ref{covlemma}, see Figures~\ref{segmentCovFig:A} and~\ref{segmentCovFig:C}.
The left/right exit edges of $X_{S,\text{in}}$ remain stationary and coincide with the left exit edges
of $X_{S,\text{lu}}$, $X_{S,\text{ru}}$, so to get an actual covering one needs to constrict the h-sets in the image in the exit direction
by a small factor.

For the two connected components of $S^{+}$ -- the ``left/right entrance'' h-sets $X_{S,\text{ls}}$, $X_{S,\text{rs}}$ one needs
to fix the entry direction height in the segment coordinates so the central direction of the segment
takes its role, while the exit direction of the segment induces the exit direction for the h-set.
Then one can prove similar theorems with backcovering relations, by reversing the vector field.

In the context of the FitzHugh-Nagumo model such relations allow us to describe the passage between
slow and fast dynamics where the periodic orbit detaches from the slow manifold and starts following a heteroclinic connection of the fast subsystem.
With an eye on this application we will state the subsequent results for a range of dimension combinations
which allows an easy proof by Lemma~\ref{covlemma}.
We suspect similar theorems hold for all dimension combinations,
and it will be a subject of further studies to formulate adequate proofs.

\begin{defn}
  Let $S$ be a segment with $u(S) = 1$ and $s(S) = s$.
  We define the h-sets:
  \begin{itemize}
    \item $X_{S,\text{lu}} \subset c_{S}^{-1}(\{-1\} \times \rr^{s} \times \rr)$ (\emph{the left exit face}),
    \item $X_{S,\text{ru}} \subset c_{S}^{-1}(\{1\} \times \rr^{s} \times \rr)$ (\emph{the right exit face})
  \end{itemize}
  as follows:
  \begin{itemize}
    \item $u( X_{S,\text{lu}} ) = u( X_{S,\text{ru}} ):= 1$ and $s( X_{S,\text{lu}} ) = s( X_{S,\text{ru}} ):= s$;
    \item we set
      \begin{equation}
        \begin{aligned}
          |X_{S, \text{lu}}| &:= c_{S}^{-1}(\{-1\} \times \rr^{s} \times \rr) \cap |S|, \\
          |X_{S, \text{ru}}| &:= c_{S}^{-1}(\{1\} \times \rr^{s} \times \rr) \cap |S|;
        \end{aligned}
      \end{equation}
    \item we identify $\{ \mp 1\} \times \rr^{s} \times \rr$ with $\rr^{s+1}$
      and then set
       \begin{equation}
        \begin{aligned}
          c_{ X_{S,\text{lu}} } &:= \rho_{u} \circ {c_{S}}_{|_{ c_{S}^{-1}(\{-1\} \times \rr^{s} \times \rr) }}, \\
          c_{ X_{S,\text{ru}} } &:= \rho_{u} \circ {c_{S}}_{|_{ c_{S}^{-1}(\{1\} \times \rr^{s} \times \rr) }},
        \end{aligned}
      \end{equation}
      where $\rho_{u}(p,q,r) = (2r-1, q)$.
  \end{itemize}
\end{defn}

In the above definition the role of $\rho_{u}$ is to change the order of coordinates, as the third center variable in $S$ becomes
an exit variable in $X_{S,\text{lu}}$ and $X_{S,\text{ru}}$.

\begin{defn}
  Let $S$ be a segment with $u(S)=u$ and $s(S)=1$.
  We define the h-sets:
  \begin{itemize}
    \item $X_{S,\text{ls}} \subset c_{S}^{-1}(\rr^{u} \times \{-1\} \times \rr)$ (\emph{the left entrance face}),
    \item $X_{S,\text{rs}} \subset c_{S}^{-1}(\rr^{u} \times \{1\} \times \rr)$ (\emph{the right entrance face})
  \end{itemize}
  as follows:
  \begin{itemize}
    \item $u( X_{S,\text{lu}} ) = u( X_{S,\text{ru}} ):=u$ and $s( X_{S,\text{ls}} ) = u( X_{S,\text{rs}} ) := 1$;
    \item we set
      \begin{equation}
        \begin{aligned}
          |X_{S, \text{ls}}| &:= c_{S}^{-1}(\rr^{u} \times \{-1\} \times \rr) \cap |S|, \\
          |X_{S, \text{rs}}| &:= c_{S}^{-1}(\rr^{u} \times \{1\} \times \rr) \cap |S|;
        \end{aligned}
      \end{equation}
    \item we identify $\rr^{u} \times \{ \mp 1\} \times \rr$ with $\rr^{u+1}$,
      then set
       \begin{equation}
        \begin{aligned}
          c_{ X_{S,\text{ls}} } &:= \rho_{s} \circ c_{S}|_{ c_{S}^{-1}(\rr^{u} \times \{-1\} \times \rr) }, \\
          c_{ X_{S,\text{rs}} } &:= \rho_{s} \circ c_{S}|_{ c_{S}^{-1}(\rr^{u} \times \{1\} \times \rr) };
        \end{aligned}
      \end{equation}
      where $\rho_{s}(p,q,r) = (p, 2r-1)$.
  \end{itemize}
\end{defn}

The role of $\rho_{s}$ is to change the center variable in $S$ to an entry variable in the h-sets $X_{S,\text{ls}}$ and $X_{S,\text{rs}}$.  

\begin{defn}
\label{def:hset-constr}
 Let $X$ be an h-set with $u(X)=u$ and $s(X)=s$ and let $\delta>0$. We define:
 \begin{itemize}
   \item the $\delta$-constricted in the exit direction h-set $X^{\delta,\text{uc}}$,
   \item the $\delta$-constricted in the entry direction h-set $X^{\delta,\text{sc}}$,
 \end{itemize}
 by setting:
 \begin{equation}
   \begin{aligned}
     c_{X^{\delta,\text{uc}}} &=  \upsilon_{\text{uc}} \circ c_{X}, \\
     c_{X^{\delta,\text{sc}}} &= \upsilon_{\text{sc}} \circ c_{X}, \\
     u(X^{\delta,\text{uc}}) = s(X^{\delta,\text{sc}}) &= u, \\
     u(X^{\delta,\text{uc}}) = s(X^{\delta,\text{sc}}) &= s;
   \end{aligned}
 \end{equation}
 where $\upsilon_{\text{uc}}, \upsilon_{\text{sc}}: \rr^{u} \times \rr^{s} \to \rr^{u} \times \rr^{s}$ and:
 \begin{equation}
   \begin{aligned}
     \upsilon_{\text{uc}}(p,q) &= ( (1+\delta)p, q ), \\
     \upsilon_{\text{sc}}(p,q) &= ( p, (1+\delta) q ).
   \end{aligned}
 \end{equation}
\end{defn}

Geometrically, $\delta$-constriction shortens the h-set by a factor $1/(1+\delta)$ in the exit/entry direction.
Our notation $\text{uc}$, $\text{sc}$ stands for constricted in the ``unstable''/``stable'' (i.e. exit/entry) direction.

\begin{thm}\label{is:3}
  Let $S$ be an isolating segment between transversal sections $\Sigma_{\text{in}}$ and $\Sigma_{\text{out}}$ with $u(S)=1$ and $s(S)=s$.
  We have the following covering relations:
  \begin{equation}
    \begin{aligned}
      X_{S,\text{in}}  &\cover{E_{S}} X^{\delta, \text{uc}}_{S, \text{lu}}, \\
      X_{S,\text{in}} &\cover{E_{S}} X^{\delta, \text{uc}}_{S, \text{ru}},
    \end{aligned}
  \end{equation}
  for all $\delta > 0$. 
\end{thm}

\begin{proof}
  We will only prove
  $X_{S,\text{in}} \cover{E_{S}}  X^{\delta, \text{uc}}_{S, \text{lu}}$, the other case is analogous.
  The idea of the proof should become immediately clear by looking at Figure~\ref{segmentCovFig:C}.
  We embed the codomain of $E_{S}$ in a folded a folded hyperplane $\Sigma_{S,u}$ 
  consisting of three parts:
  \begin{itemize}
    \item The ``upper part'' $\Sigma_{S,u}^{u} := c_{S}^{-1}\left( \{1\} \times \rr^{s} \times (-\infty, 1] \right)$;
    \item the ``middle part'' $\Sigma_{S,u}^{m} := c_{S}^{-1}\left( [-1,1] \times \rr^{s} \times \{ 1 \} \right)$;
    \item the ``lower part'' $\Sigma_{S,u}^{l} := c_{S}^{-1}\left( \{-1\} \times \rr^{s} \times (-\infty, 1] \right)$.
  \end{itemize}
  It can be regarded as a piecewise smooth section homeomorphic to $\rr^{s+1}$, transversal in the sense that there exist smooth extensions of its smooth pieces
  $\Sigma_{S,u}^{u}$, $\Sigma_{S,u}^{m}$, $\Sigma_{S,u}^{l}$ to manifolds without boundary which are transversal sections for the vector field.
  
  We equip $\Sigma_{S,u}$ with a coordinate system which is given by any homeomorphic extension
  of coordinates given on $\Sigma_{S,u}^{l}$ by $c_{X_{S,\text{lu}}}|_{\Sigma_{S,u}^{l}}$ to all $\Sigma_{S,u}$ -- we denote this extension by $c_{\Sigma_{S,u}}$.
  
  The plan is to use Lemma~\ref{covlemma} and prove conditions that give the same topological alignment 
  as needed for a covering relation.

  Recall, that by $E_{S,c}$ we denote the exit map expressed in local coordinates of the h-set $X_{S,\text{in}}$ and the section $\Sigma_{S,u}$.
  In the $\Sigma_{S,u}$ coordinates the support of $X^{\delta, \text{uc}}_{S, \text{lu}}$ is a product of two balls 
  $[\frac{-1}{1+\delta}, \frac{1}{1+\delta}] \times \overline{B_{s}(0,1)}$. To be in formal
  agreement with the definition of the support we would need to stretch out the first ball to $[-1,1]$ but it is clear that assumptions of Lemma~\ref{covlemma}
  are given by geometrical conditions which persist under such rescaling. Therefore we omit this transformation to keep the notation simple.

  By definition of $\Sigma_{S,u}$ we have
  \begin{equation}\label{eq:intersection}
    |X_{S,\text{in}}| \cap \Sigma_{S,u} = X_{S,\text{in}}^{-},
  \end{equation}
  hence $E_{S}|_{X_{S,\text{in}}^{-}} = \id_{X_{S,\text{in}}^{-}}$. Coupled with the coordinate system
  we have chosen on $\Sigma_{S,u}$ we get
  \begin{equation}\label{eq:image}
    \begin{aligned}
      E_{S,c}\left(\{-1\} \times \overline{B_{s}(0,1)}\right) &= \{-1\} \times \overline{B_{s}(0,1)}, \\
      \pi_{u}E_{S,c}\left(\{1\} \times \overline{B_{s}(0,1)}\right) &> 1.
    \end{aligned}
  \end{equation}
  This, after the aforementioned rescaling of the exit coordinate, implies Condition (C2) in Lemma~\ref{covlemma}.

  Condition (C1) follows easily. From~\eqref{eq:intersection} and (S3a) we have
  \begin{equation}\label{eq:image2}
    \pi_{s}\left(E_{S,c} ( X_{S,\text{in},c} ) \cap [-1/(1+\delta),1/(1+\delta)] \times \rr^{s}\right) \subset B_{s}(0,1),
  \end{equation}
  since we need non-zero positive time to reach $\Sigma_{S,u}$. 
\end{proof}

If we consider the exit map $E_{S^{T}}$ for the reversed flow $\dot{x}=-f(x)$ in a transposed segment $S^{T}$, we obtain the following theorem.

\begin{thm}\label{is:4}
  Let $S$ be an isolating segment between transversal sections $\Sigma_{\text{in}}$ and $\Sigma_{\text{out}}$ with $u(S)=u$ and $s(S)=1$.
  We have the following covering relations
  \begin{equation}
    \begin{aligned}
      X_{S,\text{out}}  &\backcover{E_{S^{T}}^{-1}} X^{\delta, \text{sc}}_{S, \text{ls}}, \\
      X_{S,\text{out}} &\backcover{E_{S^{T}}^{-1}} X^{\delta, \text{sc}}_{S, \text{rs}}
    \end{aligned}
  \end{equation}
  for all $\delta > 0$.
\end{thm}

\begin{thm}\label{thm:2}
  Let $\dot{x}=f(x), \ x \in \rr^{3}$ be given by a smooth vector field. Assume that there exists a sequence of transversal sections $\{ \Sigma_{i} \}_{i=0}^{k},\ k \in \mathbb{N}$,
  and sequence of h-sets
  \begin{equation}
   \mathcal{X} = \{X_{i}: \ u(X_{i})=s(X_{i})=1, \ i = 0,\dots, k \}
  \end{equation}
  such that for each two consecutive h-sets $X_{i-1}$, $X_{i} \in \mathcal{X}$ we have one of the following:
  \begin{itemize}
    \item $X_{i-1} \subset \Sigma_{i-1}$, $X_{i} \subset \Sigma_{i}$ and there exists a Poincar\'e map
      $P_{i} : \Omega_{i-1} \to \Sigma_{i}$ with $\Omega_{i-1} \subset \Sigma_{i-1}$ and an integer $w_i \in \zz^{*}$ such that
       \begin{equation}
         X_{i-1} \longgencover{P_{i},w_i} X_{i},
       \end{equation}
      \item there exists an isolating segment $S_{i}$ between $\Sigma_{i-1}$ and $\Sigma_{i}$ such that $X_{S_{i},\text{in}} = X_{i-1}$ and $X_{S_{i},\text{out}} = X_{i}$;
      \item there exists an isolating segment $S_{i}$ between $\Sigma_{i-1}$ and $\Sigma_{i}$ such that $X_{S_{i},\text{in}} = X_{i-1}$ and either
        $X_{S_{i},\text{lu}} = X_{i}$ or $X_{S_{i},\text{ru}} = X_{i}$;
      \item there exists an isolating segment $S_{i}$ between $\Sigma_{i-1}$ and $\Sigma_{i}$ such that $X_{S_{i},\text{out}} = X_{i}$ and either
        $X_{S_{i},\text{ls}} = X_{i-1}$ or $X_{S_{i},\text{rs}} = X_{i-1}$.
    \end{itemize}
  Then there exists a solution $x(t)$ of the differential equation passing consecutively through the interiors of all $X_{i}$'s.
  Moreover: 
  \begin{itemize}
    \item whenever $X_{i-1}$ and $X_{i}$ are connected by an isolating segment as its front and rear faces, respectively, the solution passes through $S^{0}_{i}$;
   \item if $X_{0} = X_{k}$ the solution $x(t)$ can be chosen to be periodic.
  \end{itemize}
\end{thm}

\begin{proof}
  First, we replace all the h-sets $X_{i}$ of the form $X_{S_{i},\text{lu}}$,\ $X_{S_{i},\text{ru}}$
  by the constricted versions $X^{\delta_{i},\text{uc}}_{S_{i},\text{lu}}$,\ $X^{\delta_{i},\text{uc}}_{S_{i},\text{ru}}$
  and the h-sets of the form $X_{S_{i},\text{ls}},\ X_{S_{i},\text{rs}}$
  by $X^{\delta_{i},\text{sc}}_{S_{i},\text{ls}}\ X^{\delta_{i},\text{sc}}_{S_{i},\text{rs}}$.
  Let us denote the new h-sets by $\tilde{X}_{i}$.
  The replacement procedure is done one by one. Each time an h-set $X_{i}$ needs to be replaced
  we choose $\delta_{i}>0$ small enough, such that
  \begin{itemize}
    \item[(1.)] any covering relation $X_{i}$ was involved in is preserved for $\tilde{X}_{i}$,
    \item[(2.)] any isolating segment that was built including $X_{i}$ as either the front or the rear face 
      can be reconstructed as an isolating segment $\tilde{S}_{i}$/$\tilde{S}_{i+1}$ with the face $\tilde{X}_{i}$.
  \end{itemize}
  It is intuitively clear that both should hold for a sufficiently small perturbation.
  To show (1.) it is enough to observe that a covering relation 
  is a $C^{0}$-open condition with respect to homeomorphisms defining the h-sets
  and persists after constricting one (or both) h-sets with $\delta$ small enough.
  The proof of such proposition would be almost the same as the proof of Theorem 13 in \cite{GideaZgliczynski}
  stating stability of covering relations under $C^{0}$ perturbations, and therefore we omit it.

  For (2.) the segment $\tilde{S}_{i}$ is constructed so that $c_{S_{i}}$ is $O(\delta_{i})$-close in the $C^{1}$ norm to $c_{\tilde{S}_{i}}$.
  We omit the details; describing the construction by precise formulas would introduce a lot of unnecessary notation.
  It is easy to see that for $\delta_{i}$ small enough the conditions (S1)-(S3) (or their counterparts) will still hold.

  We apply Theorems \ref{is:1}, \ref{is:3}, \ref{is:4} to get and a chain of covering relations
  \begin{equation}
   \tilde{X}_{0} \longgencover{g_{1},w_1} \tilde{X}_{1} \longgencover{g_{2},w_2} \tilde{X}_{2} \longgencover{g_{3},w_3} \dots \longgencover{g_{k},w_k} \tilde{X}_{k},
  \end{equation}
  where for each $g_{i}$ we have one of the following:
  \begin{itemize}
    \item $g_{i} = P_{i}$,
    \item $g_{i} = R_{i}$, $R_{i}$ given by Theorem~\ref{is:1},
    \item $g_{i} = E_{S_{i}}$,
    \item $g_{i} = E_{S^{T}_{i}}^{-1}$.
  \end{itemize}
  From here, the proof continues in the same way as the proof of Theorem~\ref{thm:1}. We obtain a sequence of points
  $\{ x_{i}: x_{i} \in \inter X_{i}, i = 1,\dots, k \}$ such that $g_{i}(x_{i-1}) = x_{i}$ and we can choose $x_{0}=x_{k}$ whenever $X_{0}=X_{k}$.
  By the same argument as in Theorem~\ref{thm:1} the sequence lies on a true trajectory of the flow; the trajectory is periodic if $x_{0}=x_{k}$.
\end{proof}

We note that the formulation of Theorem~\ref{thm:2} is not aimed at full generality. By using only Theorem~\ref{is:3} or \ref{is:4} one can produce similar theorems
when one direction is expanding and arbitrary number of directions are contracting or vice versa.

\section{A model example}\label{sec:covslowfast}

The purpose of this section is to discuss a model example for the construction of a closed chain of covering relations
and isolating segments in the FitzHugh-Nagumo equations.
Contents of this section are by no means necessary to prove the main Theorems~\ref{thm:main1}, \ref{thm:main2}, \ref{thm:main3}.
Instead, their purpose is to predict that the computer assisted proof based on Theorem~\ref{thm:2} will succeed,
so the sequence of coverings is found not just by pure luck, but is backed up by the structure of the singularly perturbed system.

Our example will be a fast-slow system of the form:
\begin{equation}\label{fastslow}
  \begin{aligned}
    \dot{x} &= f(x,y,\epsilon), \\
    \dot{y} &= \epsilon g(x,y,\epsilon),
  \end{aligned}
\end{equation}
where $x \in \rr^{2}$, $y \in \rr$, $f,g$ are smooth functions of $(x,y,\epsilon)$ and $0 < \epsilon \ll 1$ is the small parameter.
We will also write $x=(x_{1},x_{2})$ to denote the respective fast coordinates. 
By \emph{the slow manifold} of~\eqref{fastslow} we will mean the set $\{ (x,y): f(x,y,0)=0 \}$,
and by the \emph{fast subsystem} a two-dimensional system
\begin{equation}
  \dot{x} = f(x,y,0),
\end{equation}
given by fixing $y$ as a parameter.

We make the following assumptions:
\begin{itemize}
  \item[(P1)] we have two branches of the slow manifold $\Lambda_{\pm 1}$,
    that coincide with $\{ 0 \} \times \{ 1 \} \times \rr$ and $\{ 0 \} \times \{ -1 \} \times \rr$,
    respectively\footnote{One expects that the branches would actually connect with each other, but we bear in mind that this is a model example and 
    the fold points are of no interest to us. }.
    Both are hyperbolic with one expanding and one contracting direction, and the vector field in their neighborhoods $U_{\pm 1}$ is of the following form:
    \begin{align}
      f(x,y,0) &= A_{\pm 1}(y)(x \mp [0,1]^{T} ) + h_{\pm 1}(x,y), \label{P1:1} \\
      \exists \tilde{\epsilon}_{0} >0: \ 0 &< \delta_{\pm 1}  \leq \pm g(x,y,\epsilon) \leq \delta^{-1}_{\pm 1}, \ (x,y) \in U_{\pm 1}, \ \epsilon \in (0, \tilde{\epsilon}_{0}]. \label{P1:2}
    \end{align}
    The functions $A_{\pm 1},h_{\pm 1}$ are assumed to be smooth and to have the following properties
    \begin{align}
      A_{\pm 1}(y) &= \left[ \begin{array}{cc} \lambda_{u,\pm1}(y) & 0 \\ 0 & \lambda_{s,\pm 1}(y) \end{array} \right], \label{P1:3} \\
        -\delta_{\pm 1}^{-1} \leq \lambda_{s,\pm 1}(y) \leq - \delta_{\pm 1} &< 0 < \delta_{\pm 1} \leq \lambda_{u, \pm 1}(y) \leq \delta_{\pm 1}^{-1}, \label{P1:4} \\
        \frac{h_{\pm 1}(x,y)}{ \vectornorm{ x - (0,\pm 1) } } &\xrightarrow{x \to (0,\pm 1)} 0\quad \forall y. \label{P1:5}
    \end{align}
    The values $\delta_{\pm 1} > 0$ are some constant bounds, which in particular do not depend on neither $\epsilon$ nor $y$.

  \item[(P2)] For the parameterized family of the fast subsystems
    we have two parameters $y_{*},y^{*}$, without loss of generality assumed to be equal to $\mp 1$, for which there exists a transversal heteroclinic connection between
    the equilibria $(0,-1)$ and $(0, 1)$ in the first case, and $(0,1)$ and $(0,-1)$ in the second.
    That means: given any two one-dimensional transversal sections $\Sigma_{f, \pm 1}$ for the fast subsystems for $y=\pm 1$
    which have a nonempty, transversal intersection with the heteroclinic orbits, the maps $\Psi_{\pm 1}$ given by
    \begin{equation}
      \Psi_{\pm 1}: y \to W^u_{\pm 1 ,\Sigma_{f, \pm 1} }(y)- W^s_{\mp 1 ,\Sigma_{f, \pm 1}}(y) \in \rr
    \end{equation}
    have zeroes and a non-zero derivative at $y = \pm 1$.

    Here $W^u_{\pm 1, \Sigma}(y)$ and $W^s_{\pm 1, \Sigma}(y)$ denote the first intersections between the appropriate branches\footnote{To not complicate further the notation,
    we make an implicit assumption that only one pair of branches cross in each of the two subsystems and only refer to them.}
    of the unstable/stable manifolds of the equilibria $(0,\pm 1)$ with a given section $\Sigma$ in the section coordinates.

  \item[(P3)] Denote the points $(0,-1,-1)$, $(0,1,-1)$, $(0,1,1)$, $(0,-1,1)$ by $\Gamma_{\alpha}$, $\alpha \in \mathcal{I} = \{ dl,$  $ul, ur, dr \}$\footnote{
      Index letters in $\mathcal{I}$ stand for up/down and left/right and refer to positions of the points in the $(y,x_2)$ plane, see Figure~\ref{schemeFig}.}, 
      respectively and set $\epsilon=0$.
    For each $\alpha \in \mathcal{I}$ there exists a neighborhood $V_{\alpha}$ of $\Gamma_{\alpha}$,
    such that if $\Lambda_{\pm 1} \cap V_{\alpha}$ is the part of the slow manifold contained
    in $V_{\alpha}$, then the part of its unstable manifold contained in $V_{\alpha}$ coincides with the plane $\rr \times \{\pm 1\} \times \rr$,
    and the part of the stable manifold contained in $V_\alpha$ - with the plane $ \{ 0 \} \times \rr \times \rr$. Without loss of generality we can have
    $\bigcup_{\alpha \in \mathcal{I}} V_{\alpha} \subset ( U_{-1} \cup U_{1} )$.
\end{itemize}

Assumptions that provide us with straightened coordinates are used mostly to simplify the exposition.
It is our impression that the Fenichel theory, and in particular the Fenichel normal form around the slow manifold
are well-suited for verifying such conditions, see~\cite{Fenichel}.

\begin{thm}\label{thm:heuristic}
  Under assumptions (P1)-(P3), there exists an $\epsilon_{0}>0$ and
  six sets forming isolating segments for~\eqref{fastslow} for $\epsilon \in (0,\epsilon_0]$:
  \begin{itemize}
    \item $S_{u}$, $S_{d}$ - two ``long'' isolating segments positioned around the branches $\Lambda_{\pm 1}$ of the slow manifold;
    \item $S_{\alpha},\ \alpha \in \mathcal{I}$ - four short ``corner'' isolating segments, each containing the respective point $\Gamma_{\alpha}$;
  \end{itemize}
  along with the associated transversal sections of the form $\Sigma_{S_{*},\text{in}},\Sigma_{S_{*},\text{out}}$, with
  \begin{equation}
    \begin{aligned}
      u(S_{dl}) &= s(S_{dl}) = u(S_{dr}) = s(S_{dr})
      \\ &= u(S_{ul})=s(S_{ul})= u(S_{ur})=s(S_{ur})
      \\ &=u(S_{u}) = s(S_{u})=u(S_{d})=s(S_{d})
      \\ &= 1.
    \end{aligned}
  \end{equation}
  Moreover, for the h-sets defined by isolating segments we have
  \begin{align}
   X_{S_{u},\text{out}} &= X_{S_{ur},\text{in}}, \label{eq:UR} \\
   X_{S_{dr},\text{out}} &= X_{S_{d},\text{in}}, \\
   X_{S_{d},\text{out}} &= X_{S_{dl},\text{in}}, \\
   X_{S_{ul},\text{out}} &= X_{S_{u},\text{in}}, \label{eq:UL}
  \end{align}
  and the collection
  \begin{equation}
    \begin{aligned}
      \mathcal{X}_{\text{FHN},\text{P}} &= \\
      \{ &X_{S_{u},\text{in}},\ X_{S_{u},\text{out}},\ X_{S_{ur},\text{lu}},
      \ X_{S_{dr},\text{rs}}, \\
      &X_{S_{d}, \text{in}},\ X_{S_{d},\text{out}}, \ X_{S_{dl},\text{ru}},
      \ X_{S_{ul},\text{ls}},\ X_{S_{u}, \text{in}}
      \}
    \end{aligned}
  \end{equation}
  satisfies assumptions of Theorem~\ref{thm:2} for $\epsilon \in (0,\epsilon_{0}]$.
  In particular we have the following covering relations among the h-sets
  not connected by an isolating segment:
  \begin{align}
     X_{S_{dl},\text{ru}} &\cover{P_{L}}  X_{S_{ul},\text{ls}}, \\
     X_{S_{ur},\text{lu}} &\cover{P_{R}}  X_{S_{dr},\text{rs}},
  \end{align}
  where $P_{*}$ are Poincar\'e maps between the respective h-sets and transversal sections containing the next h-set.
  
  As a consequence there exists a periodic solution of the system for these parameter values.
\end{thm}

\begin{figure}
  \centering
  \begin{tikzpicture}[line cap=round,line join=round,>=latex,x=10mm,y=5mm]
 
  \draw[color=black] (-0.,0.) -- (10,0.);
  \foreach \x in {0,...,10}
  \draw[shift={(\x,0)},color=black] (0pt,1pt) -- (0pt,-1pt) node[below] {};

  \draw[color=black] (0.,0.) -- (0.,10);
  \foreach \y in {0,...,10}
  \draw[shift={(0,\y)},color=black] (2pt,0pt) -- (-1pt,0pt) node[left] {};
 

  \clip(0,-1) rectangle (10,10);

  \filldraw[draw=orange, fill=orange!25] (1.7,7.4) -- (1.7,8.6) -- (2.3,8.6) -- (2.3,7.4) -- (1.7,7.4);
  \filldraw[draw=orange, fill=orange!25] (1.6,1.2) -- (1.6,2.8) -- (2.4,2.8) -- (2.4,1.2) -- (1.6,1.2);
  \filldraw[draw=orange, fill=orange!25] (7.7,1.4) -- (7.7,2.6) -- (8.3,2.6) -- (8.3,1.4) -- (7.7,1.4);
  \filldraw[draw=orange, fill=orange!25] (7.6,7.2) -- (7.6,8.8) -- (8.4,8.8) -- (8.4,7.2) -- (7.6,7.2);

  \filldraw[draw=brown, fill=brown!25] (7.7,1.4) -- (7.7,2.6) -- (2.4,2.8) -- (2.4,1.2) -- (7.7,1.4);
  \filldraw[draw=brown, fill=brown!25] (7.6,7.2) -- (7.6,8.8) -- (2.3,8.6) -- (2.3,7.4) -- (7.6,7.2);

  \draw [color=darkgreen,semithick,domain=0.:10.] plot(\x,{8});
  \draw [color=darkgreen,semithick,domain=0.:10.] plot(\x,{2});

  \draw [->,semithick,dash pattern=on 5pt off 5pt,color=blue] (2.,2.8) -- (2.,7.4);
  \draw [->,semithick,dash pattern=on 5pt off 5pt,color=blue] (8.,7.2) -- (8.,2.6);

  \draw [->,color=blue] (4.99,8.) -- (5.01,8.);
  \draw [->,color=blue] (5.01,2.) -- (4.99,2.);
 
  \draw [color=red,thick] (1.6,2.8) -- (2.4,2.8);
  \draw [color=red,thick] (7.7,2.6) -- (8.3,2.6);
  \draw [color=red,thick] (2.4,2.8) -- (2.4,1.2);
  \draw [color=red,thick] (7.7,2.6) -- (7.7,1.4);
 
  \draw [color=red,thick] (7.6,8.8) -- (7.6,7.2);
  \draw [color=red,thick] (7.6,7.2) -- (8.4,7.2);
  \draw [color=red,thick] (2.3,7.4) -- (2.3,8.6);
  \draw [color=red,thick] (2.3,7.4) -- (1.7,7.4);
 
  \draw [color=blue,<-] (2.3-0.1,8.+0.2) ..controls(2.-0.1,8.+0.2).. (2.-0.1,7.4+0.2);
  \draw [color=blue,<-] (8.4-0.3,8.-0.6) ..controls(8.4-0.3,8.8-0.6).. (8.-0.3,8.8-0.6);
 
  \draw [color=blue,->] (8.+0.1,2.6-0.2) ..controls(8.+0.1,2.-0.2).. (7.7+0.1,2.-0.2); 
  \draw [color=blue,<-] (1.6+0.3,2.+0.6) ..controls(1.6+0.3,1.2+0.6).. (2.+0.3,1.2+0.6);

  \begin{scriptsize}
    \draw (9.2,8.) node[anchor=north] {$\Lambda_{1}$};
    \draw (9.2,2.) node[anchor=north] {$\Lambda_{-1}$};
 
    \draw (2.,5.) node[anchor=west] {$P_{L}$};
    \draw (8.,5.) node[anchor=west] {$P_{R}$};
  
    \draw (5.,8.) node[anchor=north west] {$S_{u}$};
    \draw (5.,2.) node[anchor=north west] {$S_{d}$};   

    \draw (8.18,7.88) node[anchor=north east] {$S_{ur}$};
    \draw (1.3,7.5) node[anchor=north west] {$S_{ul}$};     
    \draw (1.56,1.09) node[anchor=south west] {$S_{dl}$};
    \draw (8.7,0.7) node[anchor=south east] {$S_{dr}$};     
  
    \draw (2.,0.) node[anchor=north] {$y_{*}$};
    \draw (8.,0.) node[anchor=north] {$y^{*}$};

    \draw (5.,0.) node[anchor=north] {$y$};
    \draw (0.,5.) node[anchor=west] {$x_{1}, x_{2}$};
  \end{scriptsize}

\end{tikzpicture}
  \caption{Isolating segments and Poincar\'e maps in the model example for the periodic orbit, the sequence of h-sets plotted in red.}\label{schemeFig}
\end{figure}
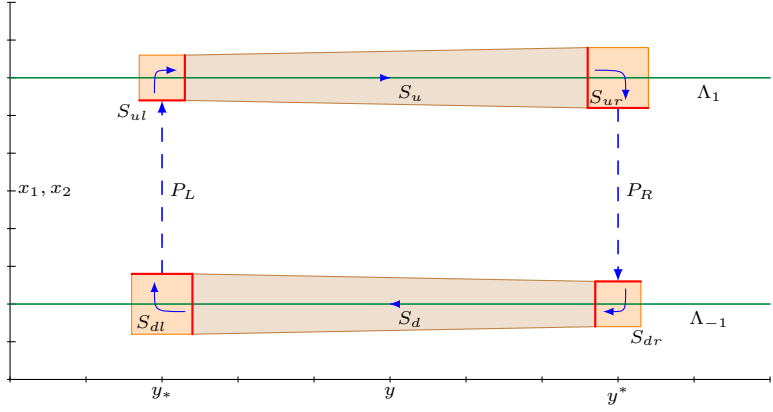

The conclusion of the theorem is portrayed in Figure~\ref{schemeFig}.
We break the proof into two parts, first we prove the existence of the corner isolating segments and coverings as a separate lemma.

\begin{lem}\label{lem:heuristic1}
  Consider the system~\eqref{fastslow}.
  For $\epsilon \in (0,\bar{\epsilon}_{0}]$, $\bar{\epsilon}_{0}>0$ small
  there exist two transversal sections of the form
  \begin{equation}
    \begin{aligned}
      \Sigma_{L} &:= \{ (x_{1},x_{2},y): \ x_{2} = 1-\varepsilon_{L} \} \cap \tilde{V}_{ul} \subset V_{ul}, \\
      \Sigma_{R} &:= \{ (x_{1},x_{2},y): \ x_{2} = -1+\varepsilon_{R} \} \cap \tilde{V}_{dr} \subset V_{dr},
    \end{aligned}
  \end{equation}
  $\tilde{V}_{ul}, \tilde{V}_{dr}$ being neighborhoods of $\Gamma_{ul}$ and $\Gamma_{dr}$
  and four isolating segments $S_{dl}$, $S_{ul}$, $S_{ur}$, $S_{dr}$ as specified in Theorem~\ref{thm:heuristic}
  such that
  \begin{equation}
    \begin{aligned}
      |X_{S_{ul},\text{ls}}| &\subset \Sigma_{L}, \\
      |X_{S_{dr},\text{rs}}| &\subset \Sigma_{R},
    \end{aligned}
  \end{equation}
  and there are coverings
  \begin{align}
    X_{S_{dl},\text{ru}} &\cover{P_{L}}  X_{S_{ul},\text{ls}}, \label{eq:PLcover} \\
    X_{S_{ur},\text{lu}} &\cover{P_{R}}  X_{S_{dr},\text{rs}}.
  \end{align}
  Moreover, the sections $\Sigma_{*}$ and the segments $S_{*}$ are $\epsilon$-independent, and 
  given a maximal diameter $\diam_{\text{max}}>0$ they can be chosen
  so that
  \begin{equation}\label{diam}
   \diam(S_{*}) < \diam_{\text{max}}.
  \end{equation}
\end{lem}

\begin{proof}
  We focus first on the ``left'' part of the picture, since all arguments for the ``right'' part are symmetric and independent.
  Without loss of generality we can assume the crossing of the unstable and stable manifolds near the point
  $\Gamma_{ul}$ occurs for $x_{2}-1$ negative and take $\varepsilon_{L}> 0$. For $\varepsilon_{L}$
  and $\epsilon$ small enough, condition (P1) implies that the linear part of the vector field dominates the higher order terms $h_{\pm 1}$, so after having
  set a sufficiently small neighborhood $\tilde{V}_{ul}$ the section $\Sigma_{L}$ is transversal.

  The construction of the isolating segments $S_{ul}, S_{dl}$ is also enabled by (P1).
  Because we already work in straightened coordinates, their supports can be chosen to be of the form:
  \begin{equation}
    \begin{aligned}
      |S_{dl}| &= [-\varepsilon_{L}, \varepsilon_{L}] \times [-1-\delta_{s, dl}, -1+\delta_{s,dl} ] \times [ -1-\delta_{u,dl}, -1+\delta_{u,dl}], \\
      |S_{ul}| &= [-\delta_{u, ul}, \delta_{u,ul} ] \times [1-\varepsilon_{L}, 1+\varepsilon_{L}] \times [ -1-\delta_{s,ul}, -1+\delta_{s,ul}], \\
      u(S_{dl}) &= s(S_{dl}) = u(S_{ul}) = s(S_{dl}) = 1,
    \end{aligned}
  \end{equation}
  where the constants $\delta_{s,dl}, \delta_{u,dl}, \delta_{s,ul}, \delta_{u,ul}$ will be fixed later in the proof.
  The changes of coordinates $c_{S_{dl}}, c_{S_{ul}}$ are defined as a translation of the cuboids to the origin of the coordinate system
  composed with rescaling to $[-1,1]^{2} \times [0,1]$. We label the first coordinate as exit, second as entry, third central.
  Again, if $\varepsilon_{L}$, $\delta_{s,dl}$, $\delta_{u,dl}$, $\delta_{u,ul}$, $\delta_{s,ul}$ are small,
  then the linear part of the vector field dominates the nonlinear part and conditions (S2b), (S3b) are satisfied for $\epsilon=0$
  and for $\epsilon>0$ small.
  Since our change of coordinates is of the form as in~\eqref{eq:centralform}, for $\epsilon>0$ small (S1a) follows from the inequalities~\eqref{P1:2}.

  We can now move on to proving the covering relation~\eqref{eq:PLcover}.
  The supports of the h-sets $X_{S_{dl},\text{ru}}$, $X_{S_{ul}, \text{ls}}$ are of the form:
  \begin{equation}
    \begin{aligned}
      |X_{S_{dl},\text{ru}}| &= \{ \varepsilon_{L} \} \times [-1-\delta_{s, dl}, -1+\delta_{s,dl} ] \times [ -1-\delta_{u,dl}, -1+\delta_{u,dl}], \\
      |X_{S_{ul}, \text{ls}}| &= [-\delta_{u, ul}, \delta_{u,ul} ] \times \{ 1-\varepsilon_{L} \} \times [ -1-\delta_{s,ul}, -1+\delta_{s,ul}].
    \end{aligned}
  \end{equation}
  In $X_{S_{dl},\text{ru}}$ the $x_{2}$ variable takes the role of the entry variable and $y$ takes the role of the exit one;
  in $X_{S_{ul}, \text{ls}}$ the variable $x_{1}$ is exit and $y$ is entry.

  Since covering relations are robust with respect to perturbations of the vector field (see Theorem 13 in \cite{GideaZgliczynski})
  it is enough to show them for $\epsilon=0$.
  From (P2) and (P3) we know that
  \begin{equation}\label{P2:mon1}
    \begin{aligned}
      P_{L}(\varepsilon_{L}, -1, -1) &= (0, 1-\varepsilon_{L}, -1), \\
      \frac{d}{dy} \pi_{x_1} P_{L}(\varepsilon_{L},-1,-1) &\neq 0,
    \end{aligned}
  \end{equation}
  and without loss of generality let us assume that $\pi_{x_1} P_{L}$ is increasing in the neighborhood of the point $(\varepsilon_{L}, -1, -1)$.
  That allows us to define two h-sets with a covering relation between them.
  The procedure is as follows:
  \begin{itemize}
    \item fix some $\delta_{s,ul}>0$.
    \item To comply with the covering condition (C1) from Lemma~\ref{covlemma} choose $\delta_{s,dl}>0$ and $\delta_{u,dl}>0$ so that
      \begin{equation}
        \pi_{y}P_{L}(|X_{S_{ul},\text{ls}}|) \subset (-1 -\delta_{s,ul}, -1 + \delta_{s,ul}).
      \end{equation}
      Now, provided $\delta_{s,dl}$ and $\delta_{u,dl}$ were chosen small enough, from~\eqref{P2:mon1}
      there exists $\varepsilon_{ul}>0$ such that
      \begin{equation}
        \begin{aligned}
          \pi_{x_{1}} P_{L}\left( \{ \varepsilon_{L} \} \times  [-1-\delta_{s, dl}, -1+\delta_{s,dl} ] \times \{ -1 -\delta_{u,dl} \} \right) &< \varepsilon_{ul} < 0,  \\
          \pi_{x_{1}} P_{L}\left( \{ \varepsilon_{L} \} \times  [-1-\delta_{s, dl}, -1+\delta_{s,dl} ] \times \{ -1 +\delta_{u,dl} \} \right) &> \varepsilon_{ul} > 0.
        \end{aligned}
      \end{equation}
    \item To fulfill (C2) it is enough to choose $\delta_{u,ul} \leq \varepsilon_{ul}$.
  \end{itemize}
  It is clear that we can choose $\varepsilon_{L}$ small enough and then perform the procedure above with $\delta$'s small
  in a way, that the diameter constriction~\eqref{diam} is satisfied.

  The same procedure is repeated for the isolating segments $S_{ur}$, $S_{dr}$; we will only introduce the notation for these segments,
  as they will be used later in the main part of the proof of Theorem~\ref{thm:heuristic}. Similarly to the left side segments
  we define them by giving the cuboid supports
  \begin{equation}
    \begin{aligned}
      |S_{ur}| &= [-\varepsilon_{R}, \varepsilon_{R}] \times [1-\delta_{s, ur}, 1+\delta_{s,ur} ] \times [ 1-\delta_{u,ur}, 1+\delta_{u,ur}], \\
      |S_{dr}| &= [-\delta_{u, dr}, \delta_{u,dr} ] \times [-1-\varepsilon_{R}, -1+\varepsilon_{R}] \times [ 1-\delta_{s,dr}, 1+\delta_{s,dr}], \\
      u(S_{ur}) &= s(S_{ur}) = u(S_{dr}) = s(S_{dr}) = 1,
    \end{aligned}
  \end{equation}
  and the coordinate changes $c_{S_{ur}}$, $c_{S_{dr}}$ are again simple translations and rescalings to $[-1,1]^{2} \times [0,1]$, so the first
  variable in the supports is the exit one and the second is entry.

  The supports of h-sets of interest $X_{S_{ur},\text{lu}}$, $X_{S_{ur}, \text{rs}}$ are as follows:
  \begin{equation}
    \begin{aligned}
      |X_{S_{ur},\text{lu}}| &= \{ -\varepsilon_{R} \} \times [1-\delta_{s, ur}, 1+\delta_{s,ur} ] \times [ 1-\delta_{u,ur}, 1+\delta_{u,ur}], \\
      |X_{S_{dr}, \text{rs}}| &= [-\delta_{u, dr}, \delta_{u,dr} ] \times \{ -1+\varepsilon_{R} \} \times [ -1-\delta_{s,dr}, -1+\delta_{s,dr}].
    \end{aligned}
  \end{equation}
  We will not go into details of determining $\delta_{u, ur}$, $\delta_{s,ur}$, $\delta_{u,dr}$, $\delta_{s,dr}$ and $\varepsilon_{R}$
  - the procedure is exactly the same as for the left side segments. The variable $y$ is the exit variable and $x_{2}$ is the entry variable in $X_{S_{ur},\text{lu}}$;
  as for $X_{S_{dr},\text{rs}}$, $x_{1}$ is the exit one and $y$ is entry.

  By taking  the minimum of all upper bounds on $\epsilon$'s throughout this lemma we obtain $\bar{\epsilon}_{0}$ and the proof is complete.
\end{proof}

We can now return to proving Theorem~\ref{thm:heuristic}. We import all the notation from the proof of the Lemma~\ref{lem:heuristic1}
and in particular assume that the isolating segments $S_{dl},S_{ul},S_{ur},S_{dr}$ and the respective h-sets can be chosen to be of the form given therein.

\begin{proof}[Proof of Theorem~\ref{thm:heuristic}]
  From Lemma~\ref{lem:heuristic1} for any given maximal corner segment diameter $\diam_{\text{max}}>0$
  we obtain a bound $\bar{\epsilon}_{0}$ on $\epsilon$'s and four isolating segments $S_{dl}, S_{ul}, S_{ur}, S_{dr}$ containing the respective points $\Gamma_{\alpha}$
  with covering relations between their respective faces.
  We set $\diam_{\text{max}}$ small enough to have
  \begin{align}
    f(x,y,0) \approx A_{1}(y)(x - [0,1]^{T} ),\ (x,y) \in \conv(|S_{ul}| \cup |S_{ur}|), \label{eq:ULin} \\
    f(x,y,0) \approx A_{-1}(y)(x + [0,1]^{T} ),\ (x,y) \in \conv(|S_{dl}| \cup |S_{dr}|),
  \end{align}
  so the higher order terms $h$ can be assumed negligible when checking isolation inequalities in these neighborhoods.

  Given our four corner isolating segments we are left with construction of
  two isolating segments $S_{u}$ and $S_{d}$ which connect the pairs $S_{ul}$, $S_{ur}$ and $S_{dr}$, $S_{dl}$, respectively.
  We will only construct $S_{u}$, the case of $S_{d}$ is analogous. The strategy is to first connect the pairs by segments,
  then, if necessary, decrease $\bar{\epsilon}_{0}$ to some smaller $\epsilon_{0}$ to obtain isolation.

  We introduce the following notation for rectangular sets around the upper branch of the slow manifold:
  \begin{equation}
    L_{u}(\delta_{u},\delta_{s},y):= [-\delta_{u},\delta_{u}] \times [1-\delta_{s},1+\delta_{s}] \times \{ y \}.
  \end{equation}

  We set
  \begin{equation}
    \begin{aligned}
      a_{u} &:= -1 + \delta_{s,ul}, \\
      b_{u} &:= 1 - \delta_{u,ur},
    \end{aligned}
  \end{equation}
  and we can assume that $a_{u}<b_{u}$.
  Now, we can define $S_{u}$ as a cuboid stretching from $X_{S_{ul,\text{out}}}$ to $X_{S_{ur,\text{in}}}$ as follows.
  For the support we put
  \begin{equation}\label{eq:SuForm}
    \begin{aligned}
    |S_{u}| &:=  \\
    &\bigcup_{\xi \in [0,1]} L_{u}\left(
                         (1-\xi)\delta_{u,ul} + \xi \delta_{u,ur},
                         (1-\xi)\delta_{s,ul} + \xi \delta_{s,ur},
                         (1-\xi)a_{u} + \xi b_{u}
                         \right).
   \end{aligned}
  \end{equation}
  We also set $u(S_{u})=s(S_{u}):=1$.
  There is no need for description of $c_{S_{u}}$ by precise formulas, so we only mention that it is a composition of
  \begin{itemize}
    \item a diffeomorphism which rescales each fiber $L_{u}(\cdot,y)$, given by fixing $y \in [a_{u},b_{u}]$, to $[-1,1] \times [-1,1]$,
    \item a rescaling in the central, $y$ direction from $[a_{u}, b_{u}]$ to $[a_{u}, a_{u}+1]$,
    \item a translation to the origin of the coordinate system.
  \end{itemize}
  As with the corner segments, $x_{1}$ is labeled as the exit direction, $x_{2}$ as entry, and $y$ as the central direction.
  Then one sees that equalities~\eqref{eq:UR} and~\eqref{eq:UL} are true.
  Condition (S1a) is a consequence of inequalities~\eqref{P1:2} for small $\epsilon$ , as the change of variables $c_{S_{u}}$
  in the central direction takes the form~\eqref{eq:centralform}. The upper bound for $\epsilon$'s
  given by $\bar{\epsilon}_{0}$ may need to be decreased at this step.

  It remains to check (S2b) and (S3b) and for that purpose we may need to further reduce $\bar{\epsilon}_{0}$.
  Normals to $S_{u}^{-}$ pointing outward of $|S_{u}|$ are given by
  \begin{equation}\label{eq:n+}
    n_{-}(x,y) = \left(\sgn x_{1}, 0, -\frac{\delta_{u,ur} - \delta_{u,ul}}{b_{u}-a_{u}} \right).
  \end{equation}
  From~\eqref{eq:ULin}, \eqref{eq:n+} and \eqref{P1:3}, \eqref{P1:4} for $(x,y) \in S_{u}^{-}$ we have
  \begin{equation}
    \begin{aligned}
    \langle (f,\epsilon g), n_{-} \rangle (x,y,\epsilon) &\approx \lambda_{u,1}(y) |x_{1}|  - \epsilon g(x,y,\epsilon) \frac{\delta_{u,ur} - \delta_{u,ul}}{b_{u}-a_{u}}  \\
    &> \delta_{1} |x_{1}| - \frac{\epsilon}{\delta_{1}} \frac{|\delta_{u,ur} - \delta_{u,ul}|}{b_{u}-a_{u}}
    \end{aligned}
  \end{equation}
  and the right-hand side is greater than 0 for $\epsilon \in (0, \bar{\epsilon}_{0}]$, $\bar{\epsilon}_{0}$ small enough, see Figure~\ref{slopeFig}.
  This proves (S2b).

  \begin{figure}
    \centering
    \begin{tikzpicture}[line cap=round,line join=round,>=latex,x=2.5mm,y=2.5mm]
 
  \draw[color=black] (-0.,0.) -- (20,0.);
  \foreach \x in {0,...,10}
  \draw[shift={(2*\x,0)},color=black] (0pt,1pt) -- (0pt,-1pt) node[below] {};

  \draw[color=black] (0.,0.) -- (0.,10);
  \foreach \y in {0,...,5}
  \draw[shift={(0,2*\y)},color=black] (2pt,0pt) -- (-1pt,0pt) node[left] {};
 
  \clip(0,-2) rectangle (40,12);

\begin{scope}[spy using outlines=
      {magnification=4, size=90pt, rounded corners, connect spies}]

  \fill [color=brown!25] (2.,4.) -- (2.,6.) -- (18.,8.) -- (18.,2.);
  \draw [color=brown] (2.,4.) -- (18.,2.);
  \draw [color=brown] (2.,6.) -- (18.,8.);
  \draw [color=brown] (2.,4.) -- (2.,6.);
  \draw [color=brown] (18.,8.) -- (18.,2.);

  \draw [->,color=brown] (10.,7.) -- (11.0,9.2);
  \draw [->,color=brown] (6.,6.5) -- (7.0,8.7);
  \draw [->,color=brown] (14.,7.5) -- (15.0,9.7);
 
  \draw [->,color=brown] (10.,3.) -- (11.0,0.8);
  \draw [->,color=brown] (6.,3.5) -- (7.0,1.3);
  \draw [->,color=brown] (14.,2.5) -- (15.0,0.3);
 
  \draw [color=darkgreen,semithick,domain=0.:20.] plot(\x,{5});
  \draw [->,color=brown] (9.99,5.) -- (10.01,5.);

  \spy [black] on (10.5,7.8) in node at (30.,5.);

\end{scope}

  \draw [->,color=black] ( 30.-10.5*4.+10.*4., 5.-7.8*4.+7.*4) -- ( 30.-10.5*4.+10.*4. , 5.-7.8*4.+9.2*4);
  \draw [->,color=black] ( 30.-10.5*4.+10.*4., 5.-7.8*4.+7.*4) -- ( 30.-10.5*4.+11.*4. , 5.-7.8*4.+7.*4);
 
  \draw [color=black,dashed] ( 30.-10.5*4.+10.*4., 5.-7.8*4.+9.2*4) -- ( 30.-10.5*4.+11.*4. , 5.-7.8*4.+9.2*4);
  \draw [color=black,dashed] ( 30.-10.5*4.+11.*4., 5.-7.8*4.+7.*4) -- ( 30.-10.5*4.+11.*4. , 5.-7.8*4.+9.2*4);

  \draw ( 30.-10.5*4.+11.*4. , 5.-7.8*4.+7.*4) node[anchor=north] {\scriptsize{$\displaystyle\frac{dy}{dt}$}};
  \draw ( 30.-10.5*4.+10.*4. , 5.-7.8*4.+9.2*4 - 1) node[anchor=east] {\scriptsize{$\displaystyle\frac{dx_{1}}{dt}$}};

  \draw (10.,0.) node[anchor=north] {\scriptsize{$y$}};
  \draw (0.,9.5) node[anchor=west] {\scriptsize{$x_{1}$}};

\end{tikzpicture}
    \caption{Isolation in segments around the slow manifold for small~$\epsilon$. The fast component of the vector field dominates the slow one
    and offsets the influence of the slope on isolation inequalities.}\label{slopeFig}
  \end{figure}
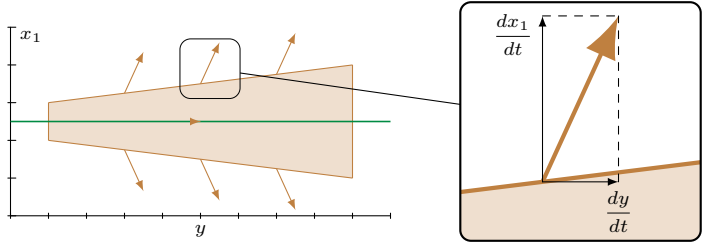

  Verifying (S3b) goes along the same lines, the expression for outward normals is
  \begin{equation}\label{eq:n-}
    n_{+}(x,y) = \left(0, \sgn(x_{2}-1), -\frac{\delta_{s,ur} - \delta_{s,ul}}{b_{u}-a_{u}} \right),
  \end{equation}
  and one readily checks that
  \begin{equation}
    \begin{aligned}
      \langle (f,\epsilon g), n_{+} \rangle (x,y,\epsilon) &\approx \lambda_{s,1}(y) |x_{2}-1|  - \epsilon g(x,y,\epsilon) \frac{\delta_{s,ur} - \delta_{s,ul}}{b_{u}-a_{u}}  \\
      &< -\delta_{1} |x_{2} - 1| + \frac{\epsilon}{\delta_{1}} \frac{|\delta_{s,ur} - \delta_{s,ul}|}{b_{u}-a_{u}}<0, \ \epsilon \in (0,\bar{\epsilon}_{0}],
    \end{aligned}
  \end{equation}
  decreasing $\bar{\epsilon}_{0}$ if necessary.
  
  The only difference in the construction of $S_{d}$ is that the recipe for $c_{S_{d}}$ has to include a flip in the $y$ direction
  so we can have $X_{S_{d},\text{in}}=X_{S_{dr},\text{out}}$ and $X_{S_{d},\text{out}}=X_{S_{dr},\text{in}}$.
  By taking minimum of all upper bounds for $\bar{\epsilon}_{0}$ throughout the proof we obtain the desired $\epsilon_{0}$.
\end{proof}

\section{Computer assisted proofs of Theorems~\ref{thm:main1}, \ref{thm:main2} and \ref{thm:main3}}\label{sec:implementation}

Most of the numerical values in this section  are given as approximations with 8 significant digits.
An exception to that are the equation parameters, which are exact.
Therefore, computations that are described below are not actually rigorous, but the programs
execute rigorous computations for values close to the ones provided.
Actual values in the program used for rigorous computations are intervals with double precision endpoints -- we decided
that writing their binary representations would obscure the exposition.
If needed, exact values can always be retrieved by the reader from the programs.
If interval is very narrow and used to represent only one particular value,
such as a coordinate of a point, we just write a single value instead.

Rigorous and nonrigorous integration, computation of enclosures of Poincar\'e maps defined between affine sections
and their derivatives, linear interval algebra and interval arithmetics is handled
by routines from the CAPD libraries~\cite{CAPD} and we do not discuss it here.
For rigorous integration we used the Taylor integrator provided in CAPD.

The source code executing the proofs is available at the author's webpage~\cite{Czechowski}.
Our exposition loosely follows what is performed by our programs.
The best way to examine the proof in detail is to look into the source code files.  
For most objects we use the same notation in the description as in the source code,
however occasionally these two differ. 
In such cases identifying the appropriate variables should be easy from the context and from the comments
left in the source code files.

For a given vector object $\texttt{x}$, by $\texttt{x[i-1]}$ we denote the i-th coordinate of $\texttt{x}$.
We will denote the right-hand side of~\eqref{FhnOde} by $F$.

We recall that the fast subsystem of~\eqref{FhnOde} is given by: 
\begin{equation}\label{eq:fastsub}
     \begin{aligned}
       u'&=v, \\
       v'&=0.2(\theta v - u(u-0.1)(1-u) + w).
     \end{aligned}
\end{equation}

Unless otherwise specified, the half-open parameter intervals $\epsilon \in (0,\epsilon_0]$ is treated 
in computations by enclosing it in a closed interval $[0,\epsilon_0]$.
The assumption $\epsilon \neq 0$ is utilized only in verification of condition (S1a) for isolating segments (see Subsection~\ref{subsubsec:segments}).

\subsection{General remarks}

\subsubsection{H-sets and covering relations}\label{subsec:covrel}

  Each h-set $X$ appearing in our program is two-dimensional with $u(X)=s(X):=1$ and can be identified with a parallelogram lying 
  within some affine section. Verification of covering relations is done exclusively by means of Lemmas~\ref{covlemma}, \ref{backcovlemma}
  (see also Remark~\ref{rem:covlemma}). 
  The procedure is relatively straightforward and has been described in detail in several papers, see for example~\cite{WilczakZgliczynski2},
  therefore we do not repeat it here. We only mention that, if needed, the procedure may include subdivision of h-sets.
  This reduces the wrapping effect, but greatly increases runtimes
  (note that wrapping is already significantly reduced by use of the Lohner algorithm within the CAPD integration routines).
  Given an h-set $X$ we want to integrate with subdivision, we introduce an integer parameter $\texttt{div}$. It indicates into how many  
  equal intervals we divide the set in each direction. For example, setting $\texttt{div}=20$ means that we 
  integrate 20 pieces of $X^{-,l}, X^{+,r}$ and 400 pieces of $|X|$ to evaluate the image of the Poincar\'e map.
  In the outlines of our proofs we will indicate the values of $\texttt{div}$ to emphasize which parts of the proof involved time-consuming computations.

  \subsubsection{Segments}\label{subsubsec:segments}

  Our segments are cuboids placed along the slow manifold $C_{0}$ so that a part of it belonging to the singular orbit is enclosed by them. 
  For each segment $S$ we have $u(S)=s(S)=1$.
  All of the segments have the property~\eqref{eq:centralform}, with 
  the slow variable $w$ serving as the central variable. Therefore, to establish (S1a) it is enough to show~\eqref{eq:sgna},
  which is equivalent to verifying either $u>w$ or $u<w$ for all points of the segment. 
  This in particular allows us to handle half-open ranges $\epsilon \in (0,\epsilon_0]$ computationally,
  as at this point we effectively factor out the small parameter.
  
  Confirming (S2b) and (S3b) is simple, as all of the faces lie in affine subspaces.
  As the exit/entry directions we take the approximate directions of the unstable/stable bundles of $C_{0}$.
  Similarly as for verification of covering relations, we subdivide the sets $S^{-}$, $S^{+}$ before evaluating isolation inequalities. 
  The normals are constant within a face, the actual benefit is in reduction of wrapping in evaluation of the right-hand side of the vector field 
  over a face.

  Our segments are rigid and the stable and unstable bundles of $C_{0}$ actually slightly revolve as we travel along the manifold branches.
  By using a single segment to cover a long piece of the branch we could not expect conditions (S2b), (S3b) to hold anymore.
  Therefore we use sequences of short segments, the position of each is well-aligned with the unstable/stable bundles of $C_{0}$ - we call them \emph{chains of segments}.
  They are simply sequences of short segments placed one after another, so a longer piece of the slow manifold can be covered. 
  We require that each segment $S_{i}$ from a chain is an isolating segment and that
  for each two consecutive segments $S_{i}$, $S_{i+1}$ in the chain the transversal section $\Sigma_{i,\text{out}}$ containing the face $S_{i,\text{out}}$
  coincides with the section $\Sigma_{i+1,\text{in}}$ containing $S_{i+1,\text{in}}$ and there is a covering relation by the identity map
  \begin{equation}\label{eq:chain}
    X_{S_{i},\text{out}} \longlongcover{\id|_{\Sigma_{i+1,\text{in}}}}  X_{S_{i+1},\text{in}}.
  \end{equation}  
  In other words, the covering relation is realized purely by the change of coordinates $c_{X_{S_{i+1},\text{in}}}\circ c_{X_{S_{i},\text{out}}}^{-1}$.
  For purposes of checking the assumptions of Theorem~\ref{thm:2}, we treat the identity map as a special case of a Poincar\'e map, see Subsection~\ref{subsec:poinc}.

  A topic we think is worth exploring, is whether chains of segments are a viable alternative to numerical integration
  in computer assisted proofs for differential inclusions
  arising from evolution PDEs; or of stiff systems where one has a good guess for the orbit
  from a nonrigorous stiff integrator. In future we plan to conduct numerical simulations
  to get more insight on that matter.
  
  \paragraph{\emph{Representation of segments}}\label{par:segments}
    
  Each segment \texttt{S} in our programs can be represented by 
  \begin{itemize}
    \item two points $\texttt{Front},\texttt{Rear} \in \rr^{3}$, serving as approximations of points on $C_{0}$,
    \item a 2x2 real matrix $\texttt{P}$ representing the rotation of the segment around the slow manifold (this does not need to be a rotation matrix) --
      it will contain approximate eigenvectors of the linearization of the fast subsystem~\eqref{eq:fastsub} at a selected point from $C_{0} \cap \texttt{S}$,
    \item four positive numbers $\texttt{a},\texttt{b},\texttt{c},\texttt{d}>0$ -- the pair $(\texttt{a},\texttt{b})$ describes how to stretch or narrow the exit and the entry widths of 
      the front face of the segment, respectively, and the pair $(\texttt{c},\texttt{d})$ does the same for the rear face.
  \end{itemize}

  For a pair of points $(a,b)$ and a 2x2 matrix $A$ we define an auxiliary linear map $\Pi_{a,b,A}: \rr^{2} \to \rr^{3}$ by
  \begin{equation}
    \Pi_{a,b,A}(x_{u},x_{s}) = \left[ {\begin{array}{c} A \\ 0 \end{array} } \right] \left[ {\begin{array}{c} ax_{u} \\ bx_{s} \end{array} } \right] .
  \end{equation}
  Our segment is then defined by
  \begin{equation}
    \begin{aligned}
      c_{\texttt{S}}^{-1}(x_{u},x_{s},x_{\mu})  &= (1-x_{\mu})  ( \texttt{Front} + \Pi_{\texttt{a},\texttt{b},\texttt{P}}(x_{u},x_{s}) )
      \\ &+ x_{\mu}  ( \texttt{Rear} + \Pi_{\texttt{c},\texttt{d},\texttt{P}}(x_{u},x_{s}) ).
   \end{aligned}
  \end{equation}
  
  For such segments one can define their front \& rear faces and the left/right entrance/exit faces $X_{\texttt{S},\text{in}}, X_{\texttt{S},\text{out}},
  X_{\texttt{S},\text{ls}}, X_{\texttt{S},\text{rs}}, X_{\texttt{S},\text{lu}}, X_{\texttt{S},\text{ru}}$ as in Subsection~\ref{subsec:segments}.

  \paragraph{Construction of chains of segments}\label{subsec:chains}
  Our recipe for creating a chain of segments  \({\bf S} \)=$\{\texttt{S}_{\texttt{i}}\}_{\texttt{i} \in \{\texttt{1}, \dots, \texttt{N} \} }$ 
  along a branch of the slow manifold is as follows.
  We assume we are given two disjoint segments $\texttt{S}_{\texttt{0}}$, $\texttt{S}_{\texttt{N+1}}$ 
  positioned along the slow manifold $C_{0}$ that we would like to connect by the chain.
  Without loss of generality we may assume that we are on the upper branch of $C_{0}$, so
  $|\texttt{S}_{\texttt{0}}|$ is to the left of $|\texttt{S}_{\texttt{N+1}}|$ in terms of the $w$ coordinate.
  
  For each segment $\texttt{S}_{\texttt{i}}$ we will use 
  its representation 
  \begin{equation}
    (\texttt{Front}_{\texttt{i}}, \texttt{Rear}_{\texttt{i}}, \texttt{P}_{\texttt{i}}, \texttt{a}_{\texttt{i}}, \texttt{b}_{\texttt{i}}, 
    \texttt{c}_{\texttt{i}}, \texttt{d}_{\texttt{i}})
  \end{equation}
  given in Paragraph~\ref{par:segments}. Wherever we mention an identity map $\id$
  between two h-sets, we mean the identity map restricted to the common transversal section.

  Our chain will connect the segments $\texttt{S}_{\texttt{0}}$, $\texttt{S}_{\texttt{N+1}}$ 
  in the sense that
  \begin{align}
    X_{\texttt{S}_{\texttt{0}}} &\cover{\id} X_{\texttt{S}_{\texttt{1}},\text{in}},\label{eq:firstcover} \\
    X_{\texttt{S}_{\texttt{N}},\text{out}} &= X_{\texttt{S}_{\texttt{N+1}},\text{in}}.\label{eq:lastcover}
  \end{align}
  We remark that we connect the faces of two segments as this is what we later do in the proof of Theorem~\ref{thm:main1}, but with little changes these 
  could as well be any two parallelogram h-sets placed on sections crossing $C_{0}$.

  Creating a chain is a sequential process akin to rigorous integration with a fixed time step; to construct the segment $\texttt{S}_{\texttt{i}}$ we
  need to know the representation of the segment $\texttt{S}_{\texttt{i-1}}$,
  If $1 \leq i<N$ we define the segment $\texttt{S}_{\texttt{i}}$ as follows
  \begin{itemize}
    \item we set $\texttt{Front}_{\texttt{i}}:= \texttt{Rear}_{\texttt{i-1}}$.
    \item The point $\texttt{Rear}_{\texttt{i}}$ is constructed by locating an (approximate) equilibrium of the fast subsystem~\eqref{eq:fastsub} with Newton's method
      for 
      \begin{equation}
        w := \texttt{Front}_{\texttt{i}}\texttt{[2]} + \texttt{1}/\texttt{N},
      \end{equation}
      and then embedding it into the 3D space by adding the value of $w$ as the third coordinate.
    \item Columns of the matrix $\texttt{P}_{\texttt{i}}$ are set as approximate eigenvectors of linearization of~\eqref{eq:fastsub} at $\texttt{Rear}_{\texttt{i}}$.
    \item For $(\texttt{a}_{\texttt{i}},\texttt{b}_{\texttt{i}})$ we put
      \begin{equation}\label{eq:shrinkAndExpand}
        \begin{aligned}
          \texttt{a}_{\texttt{i}} &:= \texttt{c}_{\texttt{i-1}}/\texttt{factor},\\
          \texttt{b}_{\texttt{i}} &:= \texttt{factor} \times \texttt{d}_{\texttt{i-1}}
        \end{aligned}
      \end{equation}
      where $\texttt{factor}$ is a real number greater than 1. In our case hardcoding $\texttt{factor}:=1.05$ gave good results.
    \item For $(\texttt{c}_{\texttt{i}},\texttt{d}_{\texttt{i}})$ we put
      \begin{equation}
        \begin{aligned}
          \texttt{c}_{\texttt{i}} &:= \frac{\texttt{i}}{\texttt{N}} \texttt{a}_{\texttt{N+1}} + \frac{\texttt{N-i}}{\texttt{N}} \texttt{c}_{\texttt{0}}, \\
          \texttt{d}_{\texttt{i}} &:= \frac{\texttt{i}}{\texttt{N}} \texttt{b}_{\texttt{N+1}} + \frac{\texttt{N-i}}{\texttt{N}} \texttt{d}_{\texttt{0}}.
        \end{aligned}
      \end{equation}
  \end{itemize}
  For the segment $\texttt{S}_{\texttt{N}}$ we proceed by the same rules with the exception that
  we set 
  \begin{equation}
    \begin{aligned}
    \texttt{Rear}_{\texttt{N}}&:= \texttt{Front}_{\texttt{N+1}}, \\
    \texttt{P}_{\texttt{N}}&:=\texttt{P}_{\texttt{N+1}},
  \end{aligned}
  \end{equation}
  to comply with~\eqref{eq:lastcover}.
  
  For such $\texttt{S}_{\texttt{i}}$ we check the conditions (S1a), (S2b), (S3b) and the covering relation
  $X_{\texttt{S}_{\texttt{i-1}},\text{out}} \cover{\id} X_{\texttt{S}_{\texttt{i+1}},\text{in}}$.
  Then, we proceed to the next segment.

  For $\texttt{N}$ large it is easy to satisfy (S2b), (S3b) for each short segment $\texttt{S}_{\texttt{i}}$,
  as each $\texttt{P}_{\texttt{i}}$ approximates the directions of the unstable and stable bundle of $C_{0}$.
  Moreover, because $\texttt{Rear}_{\texttt{i-1}}$, $\texttt{Rear}_{\texttt{i}}$ are close, for each $\texttt{i} \in \{1,\dots,\texttt{N}\}$ we have
  \begin{equation}
    \texttt{P}_{\texttt{i}} \circ \texttt{P}_{\texttt{i-1}} \approx \id .
  \end{equation}
  Thus, for the identity map in the h-sets variables we get
  \begin{equation}
    \id_{s} = c_{X_{\texttt{S}_{\texttt{i}}, \text{in}} }  \circ c_{X_{\texttt{S}_{\texttt{i-1}}, \text{out}} }^{-1}
    \approx \left[ {\begin{array}{cc} \texttt{factor} & 0 \\ 0 & \frac{1}{\texttt{factor}} \end{array} } \right],
  \end{equation}
  and there are good odds that by use of Lemma~\ref{covlemma} we can succeed in satisfying the conditions~\eqref{eq:firstcover} and~\eqref{eq:chain}.

\subsection{Proof of Theorem~\ref{thm:main1}}\label{subsec:thm11}
To deduce the existence of a periodic orbit we check the assumptions of Theorem~\ref{thm:2}.
Our strategy resembles the one given for the model example in Subsection~\ref{sec:covslowfast}, which was portrayed in Figure~\ref{schemeFig}.
The main modifications are due to numerical reasons:
\begin{itemize}
  \item we introduce two additional sections on the trajectories of the fast subsystem heteroclinics, in some distance from the corner segments,
  \item instead of the ``long'' segments $S_{u}$, $S_{d}$ we place two chains of segments along the slow manifold connecting the corner segments
    -- see Paragraph~\ref{subsec:chains}.
\end{itemize}

We divide the parameter range $\epsilon \in (0,1.5 \times 10^{-4}]$ into 
two subranges $(0,10^{-4}]$ and $[10^{-4}, 1.5 \times 10^{-4}]$. The procedure is virtually the same for both ranges
and the only reason for subdivision is that the proof would not succeed for the whole range $\epsilon \in (0,1.5 \times 10^{-4}]$ in one go, due to an accumulation of overestimates.
Following steps are executed by the program for both ranges:

\begin{enumerate}[leftmargin=*]
  \item  First, we compute four ``corner points'' 
    \begin{equation}
      \begin{aligned}
        \texttt{GammaDL} &= (-0.10841296,0,0.025044220) \approx (\Lambda_d(w_{*}),w_{*}), \\
        \texttt{GammaUL} &= (0.97034558,0,0.025044220) \approx (\Lambda_u(w_{*}),w_{*}), \\
        \texttt{GammaUR} &= (0.84174629,0,0.098807631) \approx (\Lambda_u(w^{*}),w^{*}), \\
        \texttt{GammaDR} &= (-0.23701225,0,0.098807631) \approx (\Lambda_d(w^{*}),w^{*}).
      \end{aligned}
    \end{equation}
    This computation is nonrigorous; in short we perform a shooting with $w$ procedure for the fast subsystem~\eqref{eq:fastsub}
    from first-order approximations of stable and unstable manifolds of the equilibria to an intermediate section;
    this is an approach like in~\cite{GuckenheimerKuehn}.
    The matrices given by the approximate eigenvectors of the linearization of the fast subsystem at points $\texttt{GammaDL}$,
    $\texttt{GammaUR}$, $\texttt{GammaUL}$, $\texttt{GammaDR}$ are 
    \begin{equation}
      \begin{aligned}
        \texttt{PDL} &= \texttt{PUR} = \left[ {\begin{array}{cc} 1 & 1 \\ 0.34113340 & -0.21913340 \end{array} } \right], \\
        \texttt{PUL} &= \texttt{PDR} = \left[ {\begin{array}{cc} 1 & 1 \\ 0.46313340 & -0.34113340 \end{array} } \right],
      \end{aligned}
    \end{equation}
    respectively.
  \item  We initialize four ``corner segments'' \texttt{DLSegment}, \texttt{ULSegment}, \texttt{URSegment} and \texttt{DRSegment} with data
    from Table~\ref{table:corseg} as described in Paragraph~\ref{par:segments} and check that they are isolating segments. 
    For checking the isolation formulas (S2b), (S3b) we subdivide enclosures of each of the respective faces of the exit and the entrance set into $150^{2}$ equal pieces. 
    \begin{table}[ht]
      \begin{center}
      \begin{tabular}{ | l | l | l | l |}
        \hline
        Segment &  \texttt{Front}, \texttt{Rear} & \texttt{P} & $(\texttt{a},\texttt{b}) = (\texttt{c},\texttt{d})$  \\ \hline
        \texttt{DLSegment} & $\texttt{GammaDL} \pm (0,0,0.005)$   & \texttt{PDL} & $(0.015, 0.012)$ \\ \hline
        \texttt{ULSegment} & $\texttt{GammaUL} \mp (0,0,0.005)$  & \texttt{PUL} & $(0.01, 0.015)$ \\ \hline
        \texttt{URSegment} & $\texttt{GammaUR} \mp (0,0,0.005)$   & \texttt{PUR} & $(0.029, 0.019)$ \\ \hline
        \texttt{DRSegment} & $\texttt{GammaDR} \pm (0,0,0.005)$  & \texttt{PDR} & $(0.007, 0.03)$ \\ \hline
      \end{tabular}\caption{Initialization data for the four corner segments. The pair $(\texttt{a},\texttt{b})$
      determines the exit/entry direction widths of the segments and the difference $|\texttt{Front[2]} - \texttt{Rear[2]}|$ the central
        direction width.}\label{table:corseg}
      \end{center}
    \end{table}
  \item Unlike in the model example -- Lemma~\ref{lem:heuristic1},
    the two transversal sections we integrate to are positioned within some distance from the corner segments.
    We move away from the segments because integration too close to slow manifolds poses a numerical problem -- 
    the vector field slows too much and the routines for verifying transversality fail.
    More precisely, a section $\texttt{leftSection}$ is placed on the integration path between the segments \texttt{DLSegment}, \texttt{ULSegment}
    and a section $\texttt{rightSection}$ on the path between \texttt{URSegment} and \texttt{DRSegment}. 

    We define the following Poincar\'e maps:
    \begin{itemize}
      \item $\texttt{pmDL}$ is the Poincar\'e map from $X_{\texttt{DLSegment}, \text{ru}}$ to $\texttt{leftSection}$,
      \item $\texttt{pmUL}$ is the Poincar\'e map from a subset of $\texttt{leftSection}$ to the affine section containing $X_{\texttt{ULSegment}, \text{ls}}$,
      \item $\texttt{pmUR}$ is the Poincar\'e map from $X_{\texttt{URSegment},\text{lu}}$ to $\texttt{rightSection}$,
      \item $\texttt{pmDR}$ is the Poincar\'e map from a subset of $\texttt{rightSection}$ to the affine section containing $X_{\texttt{DRSegment}, \text{rs}}$.
    \end{itemize}

    Let now us briefly describe what covering relations we verify.

    We integrate the h-set $X_{\texttt{DLSegment}, \text{ru}}$ to $\texttt{leftSection}$ and create an h-set
    $\texttt{midLeftSet} \subset \texttt{leftSection}$ so that it is $\texttt{pmUL}$-covered by a small margin by $X_{\texttt{DLSegment}, \text{ru}}$,
    see Lemma \ref{covlemma}. 
    Then, we integrate the h-set $X_{\texttt{ULSegment}, \text{ls}}$ backward in time to $\texttt{leftSection}$ and 
    verify that $\texttt{midLeftSet}$ $\texttt{pmUL}$-backcovers $X_{\texttt{ULSegment}, \text{ls}}$.

    The h-set $X_{\texttt{URSegment}, \text{lu}}$ is integrated to $\texttt{rightSection}$, and, as in the previous case, we define
    an h-set $\texttt{midRightSet} \subset \texttt{rightSection}$, such that it is $\texttt{pmUR}$-covered 
    by $X_{\texttt{URSegment}, \text{lu}}$.
    Then, we integrate the h-set $X_{\texttt{DRSegment}, \text{rs}}$ backward in time to $\texttt{rightSection}$ and 
    verify that $\texttt{midRightSet}$ $\texttt{pmDR}$-backcovers $X_{\texttt{DRSegment}, \text{rs}}$.

    Altogether, we have the following covering relations:
    \begin{equation}
      \begin{aligned}
        X_{\texttt{DLSegment},\text{ru}} &\cover{\texttt{pmDL}} \texttt{midLeftSet} \backcover{\texttt{pmUL}} X_{\texttt{ULSegment}, \text{ls}}, \\
        X_{\texttt{URSegment},\text{lu}} &\cover{\texttt{pmUR}} \texttt{midRightSet} \backcover{\texttt{pmDR}} X_{\texttt{DRSegment}, \text{rs}}.
      \end{aligned}
    \end{equation}
    Parameter $\texttt{div}$ describing partitioning of h-sets for the rigorous integration was set to 20.

  \item To close the loop, we connect the h-sets $X_{\texttt{ULSegment},\text{out}}$ and $X_{\texttt{URSegment},\text{in}}$ by a chain of segments
    \({\bf UpSegment} \) and $X_{\texttt{DRSegment},\text{out}}$ and $X_{\texttt{DLSegment},\text{in}}$ by a chain of segments \({\bf DownSegment} \)
    as described in Paragraph~\ref{subsec:chains}. 
    The number of isolating segments in each chain $\texttt{N}$ is set to 80.  
    For verification of the isolation conditions (S2b), (S3b) in each chain we partition the enclosures of each of the
    faces of their exit and entrance sets into $110^{2}$ equal pieces.
\end{enumerate}

Many choices of parameters in the program were arbitrary; of most importance are the exit/entry/central direction widths of the corner segments
given in Table~\ref{table:corseg}. 
For very small $\epsilon$ ranges (such as $\epsilon \in (0,10^{-8}]$, $\epsilon \in (0,10^{-7}]$)
various reasonable guesses would yield successful proofs, due to the eminent fast-slow structure of the equations (cf. Section~\ref{sec:covslowfast}).
However,  the range of possibilities would diminish as the upper bound on $\epsilon$ was increased,
and finding values for our final $\epsilon$ ranges was a long trial-and-error process. 
This can be explained as follows.
For large $\epsilon$'s the periodic orbit moves away from the singular orbit,
around which we position our sequence of segments and h-sets. Moreover, the hyperbolicity of the slow manifold, which plays a vital role in the creation
of the periodic orbit near the singular limit , decreases as $\epsilon$ increases.
Each time a value of a program parameter was adjusted in an attempt to succeed with a particular part of the proof,
it was possible that another part would fail. For example, increasing the central direction
widths of the corner segments facilitated the verification of covering relations for the Poincar\'e maps; but too much of an increase 
made isolation checks for the corner segments fail; increasing the exit direction widths of $\texttt{ULSegment}$, $\texttt{DRSegment}$
made the exit direction isolation checks (S2b) in segments of \({\bf UpSegment} \), \({\bf DownSegment} \)
easier to satisfy but had a negative effect on the covering relations; etc. 
It was particularly difficult to simultaneously obtain both isolation for the corner segments and covering relations in the fast regime.

By repeating the process of
\begin{itemize}
  \item trying to slightly increase the $\epsilon$ range,
  \item executing the program with given parameters,
  \item should the proof fail, changing the parameters in favor of the inequalities which were not fulfilled,
    at the cost of the ones where we still had some freedom,
\end{itemize}
we obtained a relatively large range of $\epsilon \in (0,1.5 \times 10^{-4}]$, for which the inequalities needed in our assumptions
hold by a very small margin. In particular, the right bound $1.5 \times 10^{-4}$ was large enough 
to include it in a continuation-type proof of Theorem~\ref{thm:main2}, performed in reasonable time and without using multiple precision.

It would certainly be helpful to have that procedure automated.
As one can see, we are effectively dealing with a~\emph{constraint satisfaction problem} (see~\cite{constraint})
where variables, given by the program parameters have to be chosen to satisfy constraints given by inequalities coming from covering relations
and isolation conditions.
In addition, verification of whether constraints are satisfied requires execution of the program and is fairly expensive computationally.
A suitable algorithm for adjustment of parameters to satisfy the constraints would allow to extend the range
of the small parameter $\epsilon \in (0,\epsilon_0]$ even further.
We remark that obtaining a large value of $\epsilon_0$ in this proof is crucial for achieving this parameter value 
with further validated continuation algorithms (like the one in Theorem~\ref{thm:main2})
This is due to the fact that the period of this unstable orbit is roughly proportional to $\frac{1}{\epsilon}$ 
(see Table~\ref{table:continuation}) which makes it virtually impossible to track the orbit by numerical integration methods
for very small $\epsilon$.

\subsection{Proof of Theorem~\ref{thm:main2}}\label{subsec:continuation}
Our strategy is to check the assumptions of Theorem~\ref{cov:sequence}
for a sequence of h-sets placed along a numerical approximation of an actual periodic orbit (not the singular orbit).
This can succeed for a very small range of $\epsilon$, then we need to recompute our approximation,
ending up with a continuation procedure.

We start by generating a numerical approximation vector of 212 points from the periodic orbit for $\epsilon = 0.001$ 
obtained from a nonrigorous continuation with MATCONT~\cite{MATCONT}.
From there we perform two continuation procedures, down to $\epsilon = 1.5 \times 10^{-4}$ and up to $\epsilon = 0.0015$.
Each step of the continuation consists of a routine $\texttt{proveExistenceOfOrbit}$ performed on equation~\eqref{FhnOde}
with an interval $\texttt{currentEpsRange}$ of width $\texttt{incrementSize}$ substituted for $\epsilon$.
It can be described by the following steps.

\begin{enumerate}[leftmargin=*]
  \item Given an approximation vector \( \bf initialGuess \) of $\texttt{pm\_count}$ points of the periodic orbit
    obtained from the previous continuation step (in the first step this is the MATCONT-precomputed approximation),
    we initialize a Poincar\'e section $\texttt{section}_{\texttt{i}}$ for each of the points $\texttt{initialGuess}_{\texttt{i}}$ by setting
    the origin of the section as the given point and its normal vector as the vector as the difference between
    the current and the next point of the approximation.
    Then, we refine the approximation by a nonrigorous $C^{1}$
    computation of Poincar\'e maps and their derivatives and application of Newton's method to the system of the form~\eqref{eq:problemform}.
    Note that we set the normal vector to be the difference between the current and the next point on the orbit rather than the direction
    of the vector field, as the latter can be misleading close to the strongly hyperbolic slow manifold.
    Let us denote by \( \bf correctedGuess \) the Newton-corrected approximation.
  \item Each $\texttt{section}_{\texttt{i}}$ is equipped with a coordinate system used for
    the purposes of covering by h-sets as described in Subsection~\ref{subsec:covrel}. The first column corresponding
    to the exit direction is obtained by a nonrigorous $C^{1}$ integration of any non-zero normalized vector by the variational equation of~\eqref{FhnOde}
    along the approximated orbit until it stabilizes; and then propagating it for each $\texttt{i}$ by one additional integration loop.
    Similarly, the second column (corresponding to the entry direction) is computed by backward integration of any non-zero normalized vector until it stabilizes and
    further propagation by inverse Poincar\'e maps for each $\texttt{i}$. 
    Then, we project these columns to the orthogonal complement of $\texttt{normal}_{\texttt{i}}$.
  \item Let $\texttt{pm}_{\texttt{i}}$ be a Poincar\'e map from a subset of the section 
    $\texttt{section}_{\texttt{i}}$ to the section $\texttt{section}_{\texttt{i+1}\bmod \texttt{pm\_count}}$.
    We initialize a sequence of h-sets $X_{i}$ on sections $\texttt{section}_{\texttt{i}}$
    by specifying $X_{0}$ and generating such $X_{1},\dots,X_{\texttt{pm\_count}-1}$,
    so the covering relations $X_{i} \cover{\texttt{pm}_{\texttt{i}}} X_{i+1}$, $\texttt{i} \in \{0, \dots, \texttt{pm\_count} - 2\}$ hold by a small margin.     
    The periodic orbit is strongly hyperbolic and the h-sets will quickly grow in the exit direction.
    Therefore we put an additional upper bound on the growth of the exit direction
    to prevent overestimates coming from integrating too large h-sets.
    For rigorous integration of h-sets the parameter $\texttt{div}$ was set to 5.
  \item We check that the following covering relation holds
    \begin{equation}
      X_{\texttt{pm\_count} - 1} \longlongcover{\texttt{pm}_{\texttt{pm\_count} - 1}} X_{0}.
    \end{equation}
    This implies the existence of the periodic orbit of for $\epsilon \in \texttt{currentEpsRange}$, by Theorem~\ref{cov:sequence}.
  \item We produce a new \( \bf initialGuess \) for the next step of continuation by removing the points from the approximate orbit
    where the integration time between respective sections is too short and adding them where it is too long.
    This way we can adapt the number of sections to the period of the orbit.
  \item We move the interval $\texttt{currentEpsRange}$ and proceed to the next step of the continuation.
\end{enumerate}

The continuation starts with $\texttt{incrementSize}=10^{-6}$ and the size (diameter) of the h-set $X_{0}$ of order $10^{-6}$ and both of these parameters vary throughout the proof.
If any step of $\texttt{proveExistenceOfOrbit}$ fails - for example Newton's method does not converge or there is no covering
between the h-sets, the algorithm will try to redo all the steps for a decreased $\texttt{incrementSize}$ and proportionally decrease the size of the initial h-set.
If the algorithm keeps succeeding, the program will try to increase $\texttt{incrementSize}$ and the diameter to speed up the continuation procedure.
The theorem is proved when bounds of $\texttt{currentEpsRange}$ pass the bounds of $\epsilon$ we intended to reach.
Values of $\texttt{incrementSize}$ for several different $\texttt{currentEpsRange}$ can be found in Table~\ref{table:continuation} along with
periods of the periodic orbit and amounts of sections given by $\texttt{pm\_count}$.

\begin{table}[ht]
  \begin{center}
  \scalebox{0.90}{
      \begin{tabular}{ | l | l | l | l |}
        \hline
        \texttt{currentEpsRange} & \texttt{incrementSize} &  \texttt{period} & $\texttt{pm\_count}$  \\ \hline
        $[0.0014933550, 0.001499146]$ & $5.7918161 \times 10^{-6}$ &    [201.35884, 207.17313] & 179 \\ \hline
        $[0.001, 0.001001]$ & $10^{-6}$ &  [283.37351, 292.02862] & 212 \\ \hline
        $[4.9947443, 5.0200138] \times 10^{-4}$ & $2.5269501 \times 10^{-6}$ &  [521.07987, 557.55718] & 301 \\ \hline
        $[1.5057754, 1.5132376] \times 10^{-4}$ & $7.4621539 \times 10^{-7}$ &    [1593.3303, 1846.4787] & 671  \\ \hline
      \end{tabular}
    }\caption{Sample values from the validated continuation proof of Theorem~\ref{thm:main2}. As one can see,
    the period increases significantly as $\epsilon \to 0$, making it necessary to introduce more sections and lengthening the computations. }
    \label{table:continuation}
  \end{center}
\end{table}

\subsubsection{Further continuation}\label{subsec:further}

We have decided to stop the validated continuation at $\epsilon=0.0015$.
Above that value our continuation algorithm encountered
difficulties in its nonrigorous part, and needed many manual readjustments of the continuation parameters.
As we later checked with MATCONT, this seemed not to have been caused by any bifurcation,
so, most likely, it was just a defect of our ad-hoc method of continuing approximations of the periodic orbit
by computation of Poincar\'e maps between sections.
Nonrigorous continuation methods implemented in continuation packages such as MATCONT
are based on approximation of the orbit curve by Legendre polynomials and seem more reliable than our approach.
Such a good nonrigorous approximation with a large number of collocation points
would be enough to have a rigorous part of the continuation based on Poincar\'e maps succeed,
making further continuation only a matter of computation time.
We did not implement it though, as we have decided that we are satisfied with
how wide our $\epsilon$ range is. By Theorem~\ref{thm:main3} we have already reached
the value where the standard interval Newton-Moore method applied to a sequence of Poincar\'e maps succeeds, and we think it is clear
that a proof for higher values of $\epsilon$ will pose no significant theoretical or computational challenges.

\subsection{Proof of Theorem~\ref{thm:main3}}\label{subsec:thm13}
Recall the interval Newton-Moore method for finding zeroes of a smooth map $\mathcal{F}: \rr^{n} \to \rr^{n}$:
\begin{thm}[The interval Newton-Moore method~\cite{Alefeld, Neumaier, Moore}]\label{thm:intervalNewton}
  Let $X = \Pi_{i=1}^{n} [a_{i}, b_{i}]$, $\mathcal{F}: \rr^{n} \to \rr^{n}$ be of class $C^{1}$ and let $x_{0} \in X$.
  Assume the interval enclosure of $D\mathcal{F}(X)$, denoted by $[D \mathcal{F} (X)]$ is invertible. We denote by
  \begin{equation}
    \mathcal{N}(x_{0},X):=-[D \mathcal{F} (X)]^{-1} \mathcal{F} (x_{0}) + x_{0}
  \end{equation}
  the interval Newton operator. Then
  \begin{itemize}
    \item if $\mathcal{N}(x_{0},X) \subset \inter X$, then the map $\mathcal{F}$ has a unique zero $x_{*} \in X$.
      Moreover $x_{*} \in \mathcal{N}(x_{0},X)$.
    \item If $\mathcal{N}(x_{0},X) \cap X = \emptyset$, then $\mathcal{F}(x) \neq 0$ for all $x \in X$.
  \end{itemize}
\end{thm}

We applied the interval Newton-Moore method to a
problem of the form~\eqref{eq:problemform} given by the
sequence of 179 Poincar\'e maps obtained from the last step of the continuation procedure described in Subsection~\ref{subsec:continuation},
i.e. the step, where $\texttt{currentEpsRange}$ contains 0.0015.
Let $B_{\max}(0,r)$ denote an open ball of radius $r$ centered at 0 in maximum norm.
We obtained the following inclusion
\begin{equation}
  \mathcal{N}\left( 0, \overline{B_{\max}\left(0, 10^{-6}\right)} \right) \subset B_{\max}\left(0, 4.7926638 \times 10^{-14}\right),
\end{equation}
which, by Theorem~\ref{thm:intervalNewton}, implies the existence and local uniqueness of the periodic orbit.

\begin{rem}
 We report that we have succeeded with a verified continuation based on the interval Newton-Moore method
 for the whole parameter range of Theorem~\ref{thm:main2}, that is $\epsilon \in [1.5 \times 10^{-4}, 0.0015]$.
 Although we got a little extra information on the local uniqueness of the solution of the problem~\eqref{eq:problemform}, 
 we have decided to discard this result,
 as it was vastly outperformed in terms of computation time by the method of covering relations\footnote{Substituting
 the interval Krawczyk operator for the interval Newton operator did not resolve this issue, i.e. did not allow for greater widths
 in the parameter steps.}.
 It seems that the sequential covering process in the method of covering relations
 benefits more from the strong hyperbolicity than the interval Newton operator, hence allowing to make wider steps in the parameter range for such type of problems.
 However, for ranges of higher values of $\epsilon$ the interval Newton-Moore method
 was only several times slower than the one of covering relations (e.g. $\approx 7$ times in the range $[0.001, 0.0015]$),
 so we decided to state Theorem~\ref{thm:main3} in its current form to show that
 we have achieved a parameter value where the more widespread tool is already adequate to the task.
\end{rem}

\subsection{Technical data and computation times}

All computations were performed on a laptop equipped with Intel Core i7 CPU, 1.80 GHz processor, 4GB RAM
and a Linux operating system with gcc-5.2.0. We used the 568th Subversion revision of the CAPD libraries.
The programs were not parallelized.

Verification of assumptions of Theorem~\ref{thm:main1} took 236 seconds. Over 95\% of the processor time
was taken by verification of isolation for the chains of isolating segments.

Proofs of Theorems~\ref{thm:main2} and~\ref{thm:main3} were executed by the same program.
The validated continuation in Theorem~\ref{thm:main2} was the most time consuming part -- it took 4153 seconds.
Theorem~\ref{thm:main3} is formulated for a single parameter value; the proof here was instantaneous -- it finished within 2 seconds.

We remark that the successful attempt to check the assumptions of Theorem~\ref{thm:main1} also for the range $\epsilon \in [10^{-4}, 1.5 \times 10^{-4}]$ (119 seconds)
saved us a lot of computation time. In theory we could have tried to use a validated continuation approach like in Theorem~\ref{thm:main2}
for this range. We tried it later for a subrange \nopagebreak $\epsilon \in [1.1 \times 10^{-4}, 1.5 \times 10^{-4}]$ (for the whole range 
execution of Newton's method
for the problem~\eqref{eq:problemform} within the nonrigorous part of the continuation algorithm failed due to enormous sizes of matrices to invert) 
-- it took 2571 seconds, that is over 20 times longer.
This indicates that construction of isolating segments around slow manifolds can be a valuable tool for proofs for
``regular'' parameter ranges (i.e. not including the singular perturbation parameter value) in systems with a very large separation of time scales.

\section{Concluding remarks}

We proved the existence of a periodic orbit in the FitzHugh-Nagumo equations
for a range of the small parameter $\epsilon \in (0,0.0015]$.
We also showed that the range is wide enough to reach its upper bound
with standard validated continuation and rigorous $C^{1}$ methods.
We hope that by further development of methods aimed at rigorous computations
many classical results from singular perturbation theory, such as
\begin{itemize}
  \item a proof of existence of the homoclinic orbit in the FitzHugh-Nagumo system,
  \item proofs of existence of periodic and connecting orbits for fast-slow systems with higher-dimensional slow manifolds,
  \item proofs of uniqueness and stability of the waves,
\end{itemize}
should be achievable in such explicit parameter ranges.

\section*{Acknowledgments}

This work was initiated during a visit of both authors at Uppsala University and continued throughout a visit of AC at Rutgers University.
We thank Warwick Tucker and Konstantin Mischaikow for their kind hospitality and many helpful discussions.
We are very grateful to Daniel Wilczak and Tomasz Kapela, the main developers of the CAPD library,
for the time they spend helping us with implementation of our theorems.
We also thank BIRS for hosting the workshop ``Rigorously Verified Computing for Infinite Dimensional Nonlinear Dynamics'' at the Banff Centre,
as it was a great place to exchange ideas and identify topics for future investigation in systems with multiple time scales.

\section*{Note added in proof}

After the first preprint of this article was released,
two other manuscripts on the existence of waves in explicit ranges of $\epsilon$ appeared online.
Matsue proved the existence of the periodic orbit,
the homoclinic orbit and the heteroclinic loop for the FitzHugh-Nagumo system, all without further validated continuation~\cite{Matsue}.
The methods of Matsue are computer-assisted and similar to the ones described in this paper,
though they seem to follow GSPT techniques more closely and put a bigger emphasis on tracking of invariant manifolds,
at the cost of narrower ranges of $\epsilon$.
The second work is a dissertation of AC~\cite{CzechowskiPhD}, which contains the results from this paper on the periodic orbit
and extends them to give a proof of existence of the homoclinic, also without further continuation.


\bibliography{fhn_bib}
\bibliographystyle{abbrv}



\end{document}